\documentclass{article}
\usepackage{amsmath,amsthm,amsopn,amstext,amscd,amsfonts,amssymb,verbatim}

\usepackage{color}
\usepackage{natbib}
\usepackage{graphicx}
\def\diag{\mathop{\rm diag}\nolimits}
\def\tr{\mathop{\rm tr}\nolimits}
\def\build#1#2#3{\mathrel{\mathop{#1}\limits^{#2}_{#3}}}
\def\Vec{\mathop{\rm vec}\nolimits}
\def\etr{\mathop{\rm etr}\nolimits}
\def\cov{\mathop{\rm Cov}\nolimits}
\def\var{\mathop{\rm Var}\nolimits}

\def\rank{\mathop{\rm rank}\nolimits}
\def\Half{\mathop{\rm \frac{1}{2}}\nolimits}

\renewenvironment{abstract}
                 {\vspace{6pt}
                  \begin{center}
                  \begin{minipage}{5in}
                  \centerline{\textbf{Abstract}}
                  \noindent\ignorespaces
                 }
                 {\end{minipage}\end{center}}
\newtheorem{theorem}{\textbf{Theorem}}[section]
\newtheorem{corollary}{\textbf{Corollary}}[section]

\theoremstyle{definition}
\newtheorem{definition}{\textbf{Definition}}[section]

\newtheorem{remark}{\textbf{Remark}}[section]

\title{\Large \textbf{Estimation of mean form and mean form difference under elliptical laws}}
\author{
  \textbf{Francisco J. Caro-Lopera} \thanks{Corresponding author\newline
    }\\
  {\normalsize Departament of Basic Sciences} \\
  {\normalsize Universidad de Medell\'{\i}n} \\
  {\normalsize Medell\'{\i}n, Colombia} \\
  {\normalsize E-mail: fjcaro@udem.edu.co} \\[2ex]
  \textbf{Jos\'e A. D\'{\i}az-Garc\'{\i}a} \\
  {\normalsize Universidad Aut\'onoma Agraria Antonio Narro} \\
  {\normalsize Calzada Antonio Narro 1923, Col. Buenavista} \\
  {\normalsize 25315 Saltillo, Coahuila, M\'exico}\\
  {\normalsize E-mail: jadiaz@uaaan.mx}\\
}
\date{}
\begin{document}
\maketitle

\begin{abstract}
Some ideas studied by \cite{l:93}, under a Gaussian perturbation model, are generalised in  the setting
of matrix multivariate elliptical distributions. In particular, several inaccuracies in the published
statistical perturbation model are revised. In addition, a number of aspects about identifiability and
estimability are also considered.  Instead of using  the Euclidean distance matrix for proposing
consistent estimates, this paper determines exact formulae for the moments of matrix $\mathbf{B} =
\mathbf{X}^{c}\left(\mathbf{X}^{c}\right)^{T}$, where $\mathbf{X}^{c}$ is the centered landmarks matrix.
Consistent estimation of mean form difference under elliptical laws is also studied. Finally, the main
results of the paper and some methodologies for selecting models and hypothesis testing  are applied to a
real landmark data. comparing correlation shape structure is proposed and applied in handwritten
differentiation.
\end{abstract}

\section{Introduction}\label{sec1}
Statistical theory of shape has emerged as one of the most versatile techniques of classification and
comparison of ``objects" in a number of disciplines.  By its theoretical nature, the matrix multivariate
distribution analysis fits very well into the shape analysis, but at the same time have involved strong
open problems on estimation of location and scale population parameters based on the exact distributions,
forcing the application of several less robust approaches, which were considered appropriate at first,
but later received important critics from different experts on morphometrics and related fields, see
\cite{l:93}.

Among the addressed lacks we can cite the use of asymptotic distributions, tangent plane inference,
isotropic models, Gaussian assumptions, and procrustes theory; for details of such techniques see
\cite{DM98} and the references therein.

Now, some attempts have been published recently avoiding the above restrictions and considering inference
via likelihood function using the exact shape distributions, the new theory was  termed generalised shape
theory by finding the exact shape densities indexed by families of distributions of elliptical contours.
According to the geometrical filters on shape, the resulting exact invariant distributions are expressed
in terms of series of functions termed Jack polynomials, which were uncomputable for decades, and only
recently with the works of \cite{KE06}, the individual polynomials could be computed but series of them
have involved serious problems in inference of population parameters via likelihood method. A number of
approaches with a meaningful computational success in the context of the classical Gaussian and
elliptical models are given as follows: via QR decomposition, see \cite{gm:93} and  \cite{DC14}, singular
value decompositions, see \cite{LK93}, \cite{g:91}, \cite{dgr:03}, \cite{DC12a} and \cite{dc:12}), affine
see \cite{gm:93}, \cite{dgr:03}, \cite{cdg:09}, \cite{CD12} and \cite{CDG14}), and Pseudo-Wishart, see
\cite{dc:13}. However, a feasible approach dealing with computable exact densities and a likelihood
function based on polynomials of very low degree was published recently, letting robust estimation on
location and scale population parameters very accurate; it models shapes under certain conditions via
affine transformation, which means that it removes from objects, any geometrical information of rotation,
translation, scaling and uniform shear. Meanwhile, the similarity (Euclidean) transformations via QR,
SVD, Pseudo-Wishart (invariant under rotation, translation, scaling) capture the attention of most of the
users of shape theory and is the source of the main critics.

Under Euclidean transformations the shape distributions are extremely difficult to compute and then the
associated inference, it forces the use of isotropic models, an  assumption which is unrealistic in
biology, for example,  since it says that landmarks vary independently of each other along different axes
but are correlated along a fixed axis. In fact, this isotropy assumption is very common in literature,
leaving the problem of testing solely whether shapes are equal; but for biologists, for example, they
want to identify the correlation structure of landmarks and the structures of shape which absorbs the
meaningful differences. Moreover, estimation of a full covariance structure would be the desirable
result, because, correlation among landmarks and axis is important, but the correlation among objects in
the sample should provide a complete comprehension of the involved populations, see \cite{lr:90}

Instead of estimation via likelihood method, some authors have proposed,  the Gaussian case, the
method-of-moments estimators of the mean form and the variance-covariance structure which are consistent
and simple to compute, see for example \cite{lr:91}, \cite{l:93}, \cite{rdl:02} and the references
therein. In fact, the technique was set as a critics of generalised procrustes analysis, by proving that
application of the last analysis yield inconsistent estimators of the mean form, mean shape, and
variance-covariance structure, and then all the statistical inference procedures  can produce inaccurate
results. \cite{w:01} recently reiterated the conclusions of \cite{l:93} by reporting the inability of
Procrustes methods to estimate the correct variance-covariance structure and the associated implications
for statistical inference. This aspect, is crucial because Procrustes analysis is one of the most common
method of estimation in several fields such as  morphometrics.

Given the computation problems of maximum likelihood estimators, the method-of-moments estimators emerges
as one of the promissory techniques in shape theory, however some open problems must be studied deeply.
For example,  some inaccuracies of this model presented in \cite{l:93}, assuming a matrix multivariate
Gaussian distribution must be nuanced first, and second, the method should allow non Gaussian samples,
a realistic and very common problem in morphometrics and the
usual applied areas for shape, a suitable solution comes from families of elliptical contoured
distributions, which exhibit lighter or heavier tails, or greater or less kurtosis than the Gaussian
model. Setting generalised shape theory also must include criteria for selecting models and hypothesis
testing, in order to provide an integrating theory suitable to be applied in meaningful scenarios.

Clarifying the inaccuracies of \cite{l:93} ideas, and their connection with some
theoretical studies by \cite{mn:79},  \cite{Mh:82} and \cite{dg:94} should give a unified theory setting
the isolated Gaussian approach into the general framework of the existing generalised matrix multivariate
elliptical shape theory.

Thus, estimation of mean form and mean form difference under elliptical laws is placed in this work as
follows: Section \ref{sec2} clarifies some results of the published Gaussian case and
propose the generalisation in the context of matrix multivariate elliptical distributions, it includes
some properties of matrix multivariate elliptical distribution, identifiability and estimability of the
parameters of interest, the perturbation model under a matrix multivariate elliptical distribution, and
invariance and nuisance parameters. Then Section \ref{sec3} studies the consistent estimation of the
population parameters under dependence and independence and provides exact formulae for the moments
estimators. Section \ref{sec4} provides a consistent estimation for a general non-negative definite
correlation matrix. The analysis also includes extensions to elliptical models of form difference under
the perspective of Euclidean Distance Matrix, see Section \ref{sec5}. Finally, a complete example
collecting the main results of the paper and proposing some selecting model criteria, is proposed in
Section \ref{exa}.

\section{Preliminary results}\label{sec2}

In this section we review some notation and distributional results. Also, the statistical model
to be used throughout the paper, is established and analysed. In particular some inaccuracies
of this model presented in \cite{l:93}, assuming a matrix multivariate Gaussian distribution
are corrected and then is generalised to the case where an  matrix multivariate elliptical
distribution is assumed.

\subsection{Matrix multivariate elliptical distribution}

A detailed discussion of the matrix multivariate elliptical distribution can be found for
example in \cite{fz:90} and \cite{gv:93}, among many others.

\begin{remark}
For matrix multivariate Gaussian and elliptical distributions, traditionally are used two forms
for establish that a random matrix $\mathbf{Y}$ has one of these distribution. For example in
matrix multivariate Gaussian case, this fact is written as
$$
  \mathbf{Y} \sim \mathcal{N}_{K \times D}(\boldsymbol{\mu}, \mathbf{\Sigma}, \mathbf{\Theta}),
$$
see \cite{a:81}, \cite{d:99} and \cite{fz:90} among many others authors. However, as is study
in \cite{l:93}, \cite{d:99} among others, in general the parameters $\mathbf{\Sigma}$ and
$\mathbf{\Theta}$ are not identifiable one-by-one, but $\mathbf{\Sigma} \otimes
\mathbf{\Theta}$ or $\mathbf{\Theta} \otimes \mathbf{\Sigma}$ are identifiable. Here $\otimes$
denotes the usual Kronecker product. In addition, given that $\cov(\Vec \mathbf{Y}) =
\mathbf{\Theta} \otimes \mathbf{\Sigma}$, and $\cov(\Vec \mathbf{Y}^{T}) = \mathbf{\Sigma}
\otimes \mathbf{\Theta}$, many other authors use the notation
$$
  \mathbf{Y} \sim \mathcal{N}_{K \times D}(\boldsymbol{\mu}, \mathbf{\Sigma} \otimes \mathbf{\Theta}),
$$
where ``$\Vec$" denotes the vectorisation operator, see \cite{Mh:82} and \cite{gv:93}.
Analogous situation is present for matrix multivariate elliptical distributions. We shall use
this last notation.
\end{remark}

\begin{definition}
It is say that $\mathbf{Y}$ has a matrix multivariate elliptical distribution, with location
parameter matrix $\boldsymbol{\mu} \in \Re^{K \times D}$ and  scala parameter matrix
$\mathbf{\Sigma} \otimes \mathbf{\Theta} \in \Re^{KD \times KD}$; where $\mathbf{\Sigma}$ is a
definite positive matrix, $\mathbf{\Sigma} > 0$ and $\mathbf{\Theta} > 0$, with
$\mathbf{\Sigma} \in \Re^{K \times K}$ and $\mathbf{\Theta} \in \Re^{D \times D}$, if its
density function with respect to Lebesgue measure is given by
\begin{equation}\label{ED1}
    dF_{_{\mathbf{Y}}}(\mathbf{Y}) = |\mathbf{\Sigma}|^{-D/2} |\mathbf{\Theta}|^{-K/2} h[\tr
    \mathbf{\Theta}^{-1} (\mathbf{Y}-\boldsymbol{\mu})^{T}\mathbf{\Sigma}^{-1}
    (\mathbf{Y}-\boldsymbol{\mu})] (d\mathbf{Y}),
\end{equation}
where the function  $h: \Re \rightarrow [0,\infty)$ is such that $\int_{0}^\infty
u^{KD/2-1}h(u)du < \infty$. The function $h$ is termed the \textit{density generator}. Its
characteristic function is given by
\begin{equation}\label{ecf}
    \psi_{\mathbf{Y}}(\mathbf{T}) = \etr(i \boldsymbol{\mu}^{T} \mathbf{T}) \phi(\tr \mathbf{T
  \Theta} \mathbf{T}^{T} \mathbf{\Sigma}),
\end{equation}
with $i = \sqrt{-1}$, $\phi: [0, \infty) \rightarrow \Re$ and $\etr(\cdot) = \exp (\tr
(\cdot))$. This fact is denoted as $\mathbf{Y} \sim \mathcal{E}_{K \times D}(\boldsymbol{\mu},
\mathbf{\Sigma} \otimes \mathbf{\Theta}, h)$. In addition, observe that the characteristic
function exist still when $\mathbf{\Sigma}$ and/or $\mathbf{\Theta}$ are semidefinite positive
matrices; in such case it say that $\mathbf{Y}$  has a \textit{singular matrix multivariate
elliptical distribution}, see Remark \ref{rem2} below.

In addition, note that $\cov(\Vec \mathbf{Y}) = c_{0}\mathbf{\Theta} \otimes \mathbf{\Sigma}$,
and $\cov(\Vec \mathbf{Y}^{T}) = c_{0}\mathbf{\Sigma} \otimes \mathbf{\Theta}$ where $c_{0}
=-2\phi'(0)$,
$$
  \phi'(0) = \left . \frac{d \phi(t^{2})}{dt}\right|_{t=0}.
$$
see \cite[Theorm 2.6.5, p. 62]{fz:90} and \cite[Corollary 3.2.1.1, p. 94 and Theorem 2.4.1, p.
33]{gv:93}.
\end{definition}

Is easy to see that if
$$
  \mathbf{Y} = \left (
    \begin{array}{c}
      \mathbf{Y}_{(1)}^{T} \\
      \mathbf{Y}_{(2)}^{T} \\
      \vdots \\
      \mathbf{Y}_{(k)}^{T}
    \end{array}
  \right ) = (\mathbf{Y}_{1}, \mathbf{Y}_{2}, \dots, \mathbf{Y}_{D}),
  \quad
  \boldsymbol{\mu} = \left (
    \begin{array}{c}
      \boldsymbol{\mu}_{(1)}^{T} \\
      \boldsymbol{\mu}_{(2)}^{T} \\
      \vdots \\
      \boldsymbol{\mu}_{(k)}^{T}
    \end{array}
  \right ) = (\boldsymbol{\mu}_{1}, \boldsymbol{\mu}_{2}, \dots, \boldsymbol{\mu}_{D}),
$$
$$
  \mathbf{\Sigma} = \left(
    \begin{array}{cccc}
      \sigma_{11} & \sigma_{12} & \cdots & \sigma_{1K} \\
      \sigma_{21} & \sigma_{22} & \cdots & \sigma_{2K} \\
      \vdots & \vdots & \ddots & \vdots \\
      \sigma_{K1} & \sigma_{K2} & \cdots & \sigma_{KK}
    \end{array}
    \right )
  \mbox{ and }
  \mathbf{\Theta} = \left(
    \begin{array}{cccc}
      \theta_{11} & \theta_{12} & \cdots & \theta_{1D} \\
      \theta_{21} & \theta_{22} & \cdots & \theta_{2D} \\
      \vdots & \vdots & \ddots & \vdots \\
      \theta_{D1} & \theta_{D2} & \cdots & \theta_{kD}
    \end{array}
    \right ).
$$
Then from \cite[]{fz:90}
\begin{enumerate}
  \item $\mathbf{Y}_{(i)} \sim \mathcal{E}_{D}(\boldsymbol{\mu}_{(i)},\sigma_{ii}
  \mathbf{\Theta},h), \quad i = 1,2,\dots, K,$
  \item $\mathbf{Y}_{j} \sim \mathcal{E}_{K}(\boldsymbol{\mu}_{j},\theta_{jj}
  \mathbf{\Sigma},h), \quad j = 1,2,\dots,D$,
\end{enumerate}
this is,
\begin{enumerate}
  \item $\cov(\mathbf{Y}_{(i)})= c_{0}\sigma_{ii} \mathbf{\Theta}, \quad i = 1,2,\dots, K,$
  \item $\cov(\mathbf{Y}_{j}) = c_{0}\theta_{jj} \mathbf{\Sigma}, \quad j = 1,2,\dots,D$.
\end{enumerate}
These two last affirmations are incorrect stated in \cite{l:93} in the context of the
perturbation model. For this asseveration observe that this class of matrix multivariate
elliptical distributions includes Gaussian, contaminated Gaussian, Pearson type II and VII,
Kotz, Jensen-Logistic, power exponential and Bessel distributions, among others; these
distribu\-tions have tails that are more or less weighted, and/or present a greater or smaller
degree of kurtosis than the Gaussian distribution.

\subsection{Identifiability and estimability of the parameters of interest}

Now some aspects about the identifiability and estimability of the parameters
$(\boldsymbol{\mu}, \mathbf{\Sigma} \otimes \mathbf{\Theta})$ are studied.

Note that the density (\ref{ED1}) can be write as, see \cite[p. 79]{Mh:82} and \cite[Theorem
2.1.1, p. 20]{gv:93},
\begin{small}
\begin{equation}\label{ED2}
    dF_{_{\Vec\mathbf{Y}^{T}}}(\Vec\mathbf{Y}^{T}) = |\mathbf{\Sigma} \otimes \mathbf{\Theta}|^{-1/2}
    h[\Vec^{T}(\mathbf{Y}-\boldsymbol{\mu}) \left (\mathbf{\Sigma} \otimes \mathbf{\Theta}\right)^{-1}
    \Vec(\mathbf{Y}-\boldsymbol{\mu})] (d\Vec\mathbf{Y}^{T}),
\end{equation}
\end{small}
using the fact that $ \Vec^{T}\mathbf{X}( \mathbf{DB}\otimes \mathbf{C}^{T}) \Vec\mathbf{X} =
\tr(\mathbf{BX}^{T} \mathbf{CXD})$, with $\Vec^{T}\mathbf{X} \equiv (\Vec \mathbf{X})^{T}$, and
that for $\mathbf{A} \in \Re^{n \times n}$ and $\mathbf{B} \in \Re^{m \times m}$, then
$|\mathbf{A}|^{m}|\mathbf{B}|^{n} = |\mathbf{A}\otimes \mathbf{B}|$, see \cite[Section 2.2, pp.
72-76]{Mh:82} and \cite[Section 1.4, pp. 11-13]{fz:90}. Then, denoting $\Vec\mathbf{Y}^{T} =
\mathbf{y} \in \Re^{KD}$ and $\mathbf{\Sigma} \otimes \mathbf{\Theta} = \mathbf{\Xi}$, the
density (\ref{ED2}) define the distribution of the vector $\mathbf{y}$; moreover, $\mathbf{y}
\sim \mathcal{E}_{KD}(\Vec \boldsymbol{\mu}, \mathbf{\Xi}, h)$.

Now, assume that  our data consist of a sample of matrices of size $n$ from a given population,
namely $\mathbf{Y}_{1},\mathbf{Y}_{2}, \dots, \mathbf{Y}_{n}$, and define the random matrix
$$
  \mathbb{Y} =(\Vec \mathbf{Y}_{1}^{T}, \Vec \mathbf{Y}_{2}^{T},\dots,\Vec \mathbf{Y}_{n}^{T})^{T} \in \Re^{n \times KD}.
$$
From \cite{dg:94}, assuming that $\mathbf{Y}_{1},\mathbf{Y}_{2}, \dots, \mathbf{Y}_{n}$ be
independent, the density function of $\mathbb{Y}$ admit the expression
\begin{equation}\label{ED3}
    dF_{_{\mathbb{Y}}}(\mathbb{Y}) = |\mathbf{\Sigma} \otimes \mathbf{\Theta}|^{-n/2} h[\tr
    (\mathbb{Y}-\mathbb{M})(\mathbf{\Sigma} \otimes \mathbf{\Theta})^{-1}
    (\mathbb{Y}-\mathbb{M})^{T}] (d\mathbb{Y}),
\end{equation}
where
$$
  \mathbb{M} =\mathbf{1}_{n}\Vec^{T} \boldsymbol{\mu} \in \Re^{n \times KD},
$$
and $\mathbf{1}_{n} = (1,1,\dots,1)^{T}\in \Re^{n}$; this is, $\mathbb{Y} \sim \mathcal{E}_{n
\times KD}(\mathbb{M}, \mathbf{\Sigma} \otimes \mathbf{\Theta} \otimes \mathbf{I}_{n}, h)$.
Thus, taking $p = KD$ in \cite[Theorem 4.1.1, p.129]{fz:90}, and given that $KD < n$ and
$h(\cdot)$ being nonincreasing and continuous, we have that the \textit{maximum likelihood
estimate} of $(\Vec \boldsymbol{\mu}, \mathbf{\Sigma} \otimes \mathbf{\Theta})$ is
$$
  \left (\widetilde{\Vec \boldsymbol{\mu}}, \widetilde{\mathbf{\Sigma} \otimes \mathbf{\Theta}} \right) =
  (\bar{\mathbf{y}}, \lambda_{\max}\mathbf{S}),
$$
where $\lambda_{\max}$ is the critical point where the function $h^{*}(\lambda)$ has its
maximum, with
$$
  h^{*}(\lambda) = \lambda^{-KDn/2}h(KD/\lambda),
$$
$$
  \bar{\mathbf{y}}= \frac{1}{n}\mathbb{Y}^{T}\mathbf{1}_{n} \in \Re^{KD}, \mbox{ and }
  \mathbf{S} = \mathbb{Y}^{T}\mathbf{H}_{n}\mathbb{Y} \in \Re^{KD \times KD}
$$
where $\mathbf{H}_{n} = \mathbf{I}_{n}-\displaystyle
\frac{1}{n}\mathbf{1}_{n}\mathbf{1}^{T}_{n}$, defines an orthogonal projection, this is,
$\mathbf{H}_{n} = \mathbf{H}_{n}^{T} = \mathbf{H}_{n}^{2}$ . Or alternatively
$$
  \bar{\mathbf{y}}=  \frac{1}{n}\sum_{i = 1}^{n} \Vec \mathbf{Y}_{i}^{T}, \mbox{ and }
  \mathbf{S} = \sum_{i = 1}^{n} (\Vec \mathbf{Y}_{i}^{T}- \bar{\mathbf{y}})(\Vec \mathbf{Y}_{i}^{T}-
  \bar{\mathbf{y}})^{T},
$$
from where the estimator of $\boldsymbol{\mu}$ is
$$
  \widetilde{\boldsymbol{\mu}} = \bar{\mathbf{Y}}=  \frac{1}{n}\sum_{i = 1}^{n} \mathbf{Y}_{i}.
$$
From \cite[Section 4.3]{fz:90}, several properties of the maximum likelihood estimators
$\widetilde{\boldsymbol{\mu}}$ and $\widetilde{\mathbf{\Sigma} \otimes \mathbf{\Theta}}$ are
obtained as: \textit{sufficiency, completeness, consistency and unbiasedness}. Specifically,
for $\mathbb{Y}$ with the finite 2nd moment and $h(\cdot)$ be nonincreasing and continuous,
$$
  \widehat{\boldsymbol{\mu}} = \bar{\mathbf{Y}} \quad\mbox{ and } \quad \widehat{\mathbf{\Sigma} \otimes \mathbf{\Theta}}
  =\frac{1}{2(1-n)\psi'(0)}\mathbf{S},
$$
are \textit{unbiased estimators} of $\boldsymbol{\mu}$ and $\mathbf{\Sigma} \otimes
\mathbf{\Theta}$.

\begin{remark}\label{rem2}
Observe that when $\mathbf{Y} \sim \mathcal{E}_{K \times D}(\boldsymbol{\mu}, \mathbf{\Sigma}
\otimes \mathbf{\Theta}, h)$ and its columns and/or its rows are dependent linearly, is say
that $\mathbf{Y}$ has a \textit{singular matrix multivariate elliptical distribution}. Then
$\mathbf{Y}$ has density with respect to \textit{Hasusdorff measure}. Moreover, such dependent
linearly among its columns or its rows is archived in the rank of $\mathbf{\Sigma}$ and/or
$\mathbf{\Theta}$ matrices and is denoted as: $\mathbf{Y} \sim \mathcal{E}_{K \times
D}^{s,r}(\boldsymbol{\mu}, \mathbf{\Sigma} \otimes \mathbf{\Theta}, h)$, where $s =
\rank(\mathbf{\Sigma}) \leq K$ and $r = \rank(\mathbf{\Theta}) \leq D$, see \cite[Definition
2.1.1, p. 19]{gv:93},\cite{dggf:05} and \cite{dggj:06}. As in the singular matrix multivariate
Gaussian case, the maximum likelihood estimators in singular matrix multivariate elliptical
models remain valid, see \cite{k:68} and \cite[Section 8a.5, pp.528-532]{r:73}.
\end{remark}

\subsection{Perturbation model under a matrix multivariate elliptical distribution}\label{sec2.3}

Let $\mathbf{X} \in \Re^{K \times D}$ a  random matrix representing the geometrical figure
comprising $K$ landmark, or labeled, points of dimension $D$, such that $K > D$. This matrix
$\mathbf{X}$ is termed \textit{landmark coordinate matrix}, see \cite{l:93}.

Let $\mathbf{X}_{1},\mathbf{X}_{2}, \dots, \mathbf{X}_{n}$ be a independently sample of size
$n$ of landmark coordinate matrices $\mathbf{X}_{i} \in \Re^{K \times D}$, $i = 1,2,\dots,n$,
from a given population.

The statistical model to be considered in this work is a generalisation of the perturbation
model used by \cite{l:93} among others authors. Let $\boldsymbol{\mu} \in \Re^{K \times D}$
corresponding to the mean form. Let
\begin{equation}\label{mm1}
    \mathbf{X}_{i} = (\boldsymbol{\mu} + \mathbf{E}_{i})\mathbf{\Gamma}_{i} + \mathbf{t}_{i},
    \quad, i = 1,2,\dots,n,
\end{equation}
where $\mathbf{E}_{i} \sim \mathcal{E}_{K \times D}(\mathbf{0}, \mathbf{\Sigma}_{K} \otimes
\mathbf{\Sigma}_{D}, h)$, $\mathbf{\Gamma}_{i} \in \Re^{D \times D}$ are orthogonal matrices
representing rotation and/or reflection of $(\boldsymbol{\mu} + \mathbf{E}_{i})$, and
$\mathbf{t}_{i} \in \Re^{K \times D}$ are matrices such that $\mathbf{t}_{i} =
\mathbf{1}_{k}\mathbf{a}_{i}^{T}$ representing translation, for some $\mathbf{a}_{i} \in
\Re^{D}$. From \cite[eq. (3.3.10), p. 103]{fz:90} or \cite[Theorem 2.1.2, p. 20]{gv:93} we have
\begin{equation}\label{DM}
    \mathbf{X}_{i} \sim \mathcal{E}_{K \times D}(\boldsymbol{\mu}\mathbf{\Gamma}_{i}+\mathbf{t}_{i},
    \mathbf{\Sigma}_{K} \otimes \mathbf{\Gamma}^{T}_{i}\mathbf{\Sigma}_{D}\mathbf{\Gamma}_{i}, h), \quad
    i = 1,2,\dots,n.
\end{equation}
Parameters of interest are $(\boldsymbol{\mu}, \mathbf{\Sigma}_{K} \otimes
\mathbf{\Sigma}_{D})$ and $(\mathbf{\Gamma}^{T}_{i}, \mathbf{t}_{i})$ $i = 1, 2, . . . , n$ are
the nuisance parameters. An detail explained of this perturbation model is given in \cite{l:93}
among others.

Alternatively, the model (\ref{mm1}) can be write as:
\begin{equation*}\label{mm2}
    \Vec\mathbb{X}^{T} = \diag(\mathbb{G})\Vec(\mathbb{M} + \mathbb{E})^{T} + \Vec \mathbb{T}^{T},
\end{equation*}
with
$$
  \diag(\mathbb{G}) = \left (
     \begin{array}{cccc}
       \mathbf{I}_{K} \otimes \mathbf{\Gamma}_{1}^{T} & \mathbf{0} & \cdots & \mathbf{0} \\
       \mathbf{0} & \mathbf{I}_{K} \otimes \mathbf{\Gamma}_{2}^{T} & \cdots & \mathbf{0} \\
       \vdots & \vdots & \ddots & \vdots \\
       \mathbf{0} & \mathbf{0} & \cdots & \mathbf{I}_{K} \otimes \mathbf{\Gamma}_{n}^{T}
     \end{array}
  \right )
$$
or the model (\ref{mm1}) may be rewritten in the form
\begin{equation*}\label{mm3}
    \mathbb{X} = \diag(\mathbb{M} + \mathbb{E})\mathbb{G}^{T} + \mathbb{T},
\end{equation*}
with
$$
  \diag(\mathbb{\mathbb{M} + \mathbb{E}}) = \left (
     \begin{array}{cccc}
       \Vec^{T} (\boldsymbol{\mu} + \mathbf{E}_{1})^{T} & \mathbf{0} & \cdots & \mathbf{0} \\
       \mathbf{0} & \Vec^{T} (\boldsymbol{\mu} + \mathbf{E}_{2})^{T}  & \cdots & \mathbf{0} \\
       \vdots & \vdots & \ddots & \vdots \\
       \mathbf{0} & \mathbf{0} & \cdots & \Vec^{T} (\boldsymbol{\mu} + \mathbf{E}_{n})^{T}
     \end{array}
  \right )
$$
where
$$
  \mathbb{X} = \left (
    \begin{array}{c}
      \Vec^{T}\mathbf{X}_{1}^{T} \\
      \Vec^{T}\mathbf{X}_{2}^{T} \\
      \vdots \\
      \Vec^{T}\mathbf{X}_{n}^{T}
    \end{array}
  \right ),
  \
  \mathbb{M} = \mathbf{1}_{n}\Vec^{T} \boldsymbol{\mu}^{T},
  \
  \mathbb{E} = \left (
    \begin{array}{c}
      \Vec^{T}\mathbf{E}_{1}^{T} \\
      \Vec^{T}\mathbf{E}_{2}^{T} \\
      \vdots \\
      \Vec^{T}\mathbf{E}_{n}^{T}
    \end{array}
  \right )
  \
  \mathbb{T} = \left (
    \begin{array}{c}
      \Vec^{T}\mathbf{t}_{1}^{T} \\
      \Vec^{T}\mathbf{t}_{2}^{T} \\
      \vdots \\
      \Vec^{T}\mathbf{t}_{n}^{T}
    \end{array}
  \right ),
$$
and $\mathbb{G} = \left(\mathbf{I}_{K} \otimes \mathbf{\Gamma}_{1}^{T}|\mathbf{I}_{K} \otimes
\mathbf{\Gamma}_{2}^{T}|\cdots|\mathbf{I}_{K} \otimes \mathbf{\Gamma}_{n}^{T}\right )$ where
$$
  \mathbb{E} \sim \mathcal{E}_{n \times KD}(\mathbf{0}, \mathbf{I}_{n} \otimes \mathbf{\Sigma}_{K}
  \otimes \mathbf{\Sigma}_{D}, h),
$$
or
$$
  \Vec \mathbb{E}^{T} \sim \mathcal{E}_{nKD}(\Vec \mathbf{0}, \mathbf{I}_{n} \otimes \mathbf{\Sigma}_{K}
  \otimes \mathbf{\Sigma}_{D}, h),
$$
Hence
$$
  \Vec\mathbb{X}^{T} \sim \mathcal{E}_{nKD}\left(\diag(\mathbb{G})\Vec\mathbb{M}^{T} + \Vec\mathbb{T}^{T},
  \diag(\mathbb{G})(\mathbf{I}_{n} \otimes \mathbf{\Sigma}_{K}\otimes \mathbf{\Sigma}_{D})
  \diag(\mathbb{G})^{T}, h\right).
$$
Note that, recalling that for $\mathbf{x}$ and $\mathbf{y}$ vectors,
$\Vec(\mathbf{y}\mathbf{x}^{T}) = \mathbf{x} \otimes \mathbf{y}$, then
\begin{eqnarray*}
  \diag(\mathbb{G})\Vec\mathbb{M}^{T} &=&\left (
     \begin{array}{cccc}
       \mathbf{I}_{K} \otimes \mathbf{\Gamma}_{1}^{T} & \mathbf{0} & \cdots & \mathbf{0} \\
       \mathbf{0} & \mathbf{I}_{K} \otimes \mathbf{\Gamma}_{2}^{T} & \cdots & \mathbf{0} \\
       \vdots & \vdots & \ddots & \vdots \\
       \mathbf{0} & \mathbf{0} & \cdots & \mathbf{I}_{K} \otimes \mathbf{\Gamma}_{n}^{T}
     \end{array}
  \right )
  \left (\mathbf{1}_{n} \otimes \Vec \boldsymbol{\mu}^{T} \right )  \\
    &=& \left(
    \begin{array}{c}
      (\mathbf{I}_{K} \otimes \mathbf{\Gamma}_{1}^{T})\Vec \boldsymbol{\mu}^{T} \\
      (\mathbf{I}_{K} \otimes \mathbf{\Gamma}_{2}^{T})\Vec \boldsymbol{\mu}^{T} \\
      \vdots \\
      (\mathbf{I}_{K} \otimes \mathbf{\Gamma}_{n}^{T})\Vec \boldsymbol{\mu}^{T}
    \end{array}
  \right ) \\
   &=& \left(
    \begin{array}{c}
      \Vec(\boldsymbol{\mu}\mathbf{\Gamma}_{1})^{T} \\
      \Vec(\boldsymbol{\mu}\mathbf{\Gamma}_{2})^{T} \\
      \vdots \\
      \Vec(\boldsymbol{\mu}\mathbf{\Gamma}_{n})^{T}
    \end{array}
  \right ),
\end{eqnarray*}
and $\diag(\mathbb{G})(\mathbf{I}_{n} \otimes \mathbf{\Sigma}_{K}\otimes \mathbf{\Sigma}_{D})
\diag(\mathbb{G})^{T}$ is
\begin{eqnarray*}
   &=& = \left (
    \begin{array}{cccc}
      \mathbf{\Sigma}_{K}\otimes \mathbf{\Gamma}_{1}^{T}\mathbf{\Sigma}_{D}\mathbf{\Gamma}_{1} & \mathbf{0} & \cdots & \mathbf{0} \\
      \mathbf{0} & \mathbf{\Sigma}_{K}\otimes \mathbf{\Gamma}_{2}^{T}\mathbf{\Sigma}_{D}\mathbf{\Gamma}_{2} & \cdots & \mathbf{0} \\
      \vdots & \vdots & \ddots & \vdots \\
      \mathbf{0} & \mathbf{0} & \cdots & \mathbf{\Sigma}_{K}\otimes \mathbf{\Gamma}_{n}^{T}\mathbf{\Sigma}_{D}\mathbf{\Gamma}_{n}
    \end{array}
  \right ) \\
   &=& \sum_{i = 1}^{n} \mathbf{E}_{ii}^{n} \otimes \mathbf{\Sigma}_{K}\otimes \mathbf{\Gamma}_{i}^{T}
   \mathbf{\Sigma}_{D}\mathbf{\Gamma}_{i},
\end{eqnarray*}
where if $\mathbf{e}^{n}_{i}$ the $i$th column unit vector of order $n$, then
$\mathbf{E}_{ii}^{n} = \mathbf{e}^{n}_{i}(\mathbf{e}^{n}_{i})^{T}$.

Finally observe that
$$
  E(\Vec\mathbb{X}^{T}) =
  \left(
    \begin{array}{c}
      \Vec(\boldsymbol{\mu}\mathbf{\Gamma}_{1})^{T} \\
      \Vec(\boldsymbol{\mu}\mathbf{\Gamma}_{2})^{T} \\
      \vdots \\
      \Vec(\boldsymbol{\mu}\mathbf{\Gamma}_{n})^{T}
    \end{array}
  \right ) + \Vec \mathbb{T}^{T}= \sum_{i = 1}^{n} \mathbf{e}_{i}^{n} \otimes \left(
  \Vec(\boldsymbol{\mu}\mathbf{\Gamma}_{i})^{T} + \Vec \mathbf{t}_{i}^{T}\right)
$$
Then
$$
  E(\mathbb{X}) = \sum_{i = 1}^{n} \mathbf{e}_{i}^{n} \left(
  \Vec(\boldsymbol{\mu}\mathbf{\Gamma}_{i})^{T} + \Vec \mathbf{t}_{i}^{T}\right)^{T}.
$$
Therefore
$$
  \mathbb{X} \sim \mathcal{E}_{n \times KD}\left(\sum_{i = 1}^{n} \mathbf{e}_{i}^{n} \left(
  \Vec(\boldsymbol{\mu}\mathbf{\Gamma}_{i})^{T} + \Vec \mathbf{t}_{i}^{T}\right)^{T},
  \sum_{i = 1}^{n} \mathbf{E}_{ii}^{n} \otimes \mathbf{\Sigma}_{K}\otimes \mathbf{\Gamma}_{i}^{T}
   \mathbf{\Sigma}_{D}\mathbf{\Gamma}_{i}, h\right).
$$

\subsection{Invariance and nuisance parameters}

In general, when a model contains nuisance parameters, the first step is to remove them. As in
the matrix multivariate Gaussian model considered by \cite{l:93}, under an matrix multivariate
elliptical model this objective is achieved through a simple transformation.

From (\ref{DM})
$$
  \mathbf{X}_{i} \sim \mathcal{E}_{K \times D}(\boldsymbol{\mu}\mathbf{\Gamma}_{i}+ \mathbf{t}_{i},
  \mathbf{\Sigma}_{K} \otimes \mathbf{\Gamma}^{T}_{i}\mathbf{\Sigma}_{D}\mathbf{\Gamma}_{i}, h), \quad
  i = 1,2,\dots,n.
$$
Recalling that $\mathbf{H}_{K}\mathbf{1}_{K} = \mathbf{0}_{k}$ and $\mathbf{1}_{K}^{T}
\mathbf{H}_{K} = \mathbf{0}_{k}^{T}$, then, defining $\mathbf{X}_{i}^{c} =
\mathbf{H}_{K}\mathbf{X}_{i}$, we have
\begin{equation}\label{DXc}
    \mathbf{X}_{i}^{c} \sim \mathcal{E}_{K \times D}^{(K-1),D}(\boldsymbol{\mu}^{*}\mathbf{\Gamma}_{i},
  \mathbf{\Sigma}_{K}^{*} \otimes \mathbf{\Gamma}^{T}_{i}\mathbf{\Sigma}_{D}\mathbf{\Gamma}_{i}, h), \quad
  i = 1,2,\dots,n.
\end{equation}
where $\boldsymbol{\mu}^{*} = \mathbf{H}_{K}\boldsymbol{\mu}$ and $\mathbf{\Sigma}_{K}^{*} =
\mathbf{H}_{K}\mathbf{\Sigma}_{K}\mathbf{H}_{K}$, $\mathbf{H}_{K}\mathbf{t}_{i} =
\mathbf{H}_{K}\mathbf{1}_{K}\mathbf{a}_{i}^{T} =\mathbf{0}$ for all $i = 1,2,\dots,n$, and
$\boldsymbol{\mu}^{*}$ is such that its columns sum to zero, that is, it is a centered matrix .

Given that $K > D$ and that $\rank(\mathbf{\Sigma}_{K}^{*}) = K-1$, from \cite{dggf:05} and
\cite{dggj:06} we have that
\begin{equation}\label{dBi}
    \mathbf{B}_{i} = \mathbf{X}_{i}^{c}(\mathbf{\Gamma}^{T}_{i}\mathbf{\Sigma}_{D}\mathbf{\Gamma}_{i})^{-1}
  (\mathbf{X}_{i}^{c})^{T} \sim \mathcal{GPW}_{K}^{q}(D, \mathbf{\Sigma}_{K}^{*}, \mathbf{\Sigma}_{D},
  \mathbf{\Omega}, h), \quad  i = 1,2,\dots,n.
\end{equation}
where
$$
  \mathbf{\Omega} = (\mathbf{\Sigma}_{K}^{*})^{-}\boldsymbol{\mu}^{*}\mathbf{\Gamma}_{i}
  (\mathbf{\Gamma}^{T}_{i}\mathbf{\Sigma}_{D}\mathbf{\Gamma}_{i})^{-1}
  \mathbf{\Gamma}_{i}^{T}(\boldsymbol{\mu}^{*})^{T} =
  (\mathbf{\Sigma}_{K}^{*})^{-}\boldsymbol{\mu}^{*}\mathbf{\Sigma}_{D}^{-1}(\boldsymbol{\mu}^{*})^{T}.
$$
$q = \min((K-1), D)$ and $\mathbf{A}^{-}$ is any symmetric generalised inverse of $\mathbf{A}$
such that $\mathbf{A}\mathbf{A}^{-}\mathbf{A} = \mathbf{A} = \mathbf{A}^{T}$. This is,
$\mathbf{B}_{i}$ has a \textit{generalised singular pseudo-Wishart distribution}, which is
independent of noise parameters.
\begin{remark}
Observe that $\mathbf{B}_{i}$ can be write as
$$
  \mathbf{B}_{i} = \mathbf{X}_{i}^{c}(\mathbf{\Gamma}^{T}_{i}\mathbf{\Sigma}_{D}
  \mathbf{\Gamma}_{i})^{-1}(\mathbf{X}_{i}^{c})^{T} =  \mathbf{X}_{i}^{c}\mathbf{\Gamma}^{T}_{i}
  \mathbf{\Sigma}_{D}^{-1}\mathbf{\Gamma}_{i}(\mathbf{X}_{i}^{c})^{T}
  = \mathbf{Y}_{i}\mathbf{\Sigma}_{D}^{-1}\mathbf{Y}_{i}^{T}
$$
where $\mathbf{Y}_{i} = \mathbf{X}_{i}^{c}\mathbf{\Gamma}^{T}_{i}$ and is such that
$$
  \mathbf{Y}_{i} \sim \mathcal{E}_{K \times D}^{(K-1),D}(\boldsymbol{\mu}^{*},
  \mathbf{\Sigma}_{K}^{*} \otimes \mathbf{\Sigma}_{D}, h), \quad
  i = 1,2,\dots,n.
$$
\end{remark}

In particular if $\mathbf{\Sigma}_{D} = \mathbf{I}_{D}$  and
$$
  \mathbf{X}_{i}^{c} = (\mathbf{X}_{1,i}^{c}|\mathbf{X}_{2,i}^{c}| \cdots |\mathbf{X}_{D,i}^{c} )
$$
with
$$
  \mathbf{X}_{d,i}^{c} \sim \mathcal{E}_{K }^{(K-1)}(\boldsymbol{\mu}^{*}\mathbf{\Gamma}_{i}\mathbf{e}_{d}^{K},
  \mathbf{\Sigma}_{K}^{*}, h), \quad d = 1,2,\dots,D; \ i = 1,2,\dots,n,
$$
we have that,
$$
    \mathbf{B}_{i} = \mathbf{X}_{i}^{c}(\mathbf{X}_{i}^{c})^{T} = \sum_{d = 1}^{D} \mathbf{X}_{d,i}^{c}
    (\mathbf{X}_{d,i}^{c})^{T},
$$
furthermore,
\begin{equation}\label{dBii}
  \mathbf{B}_{i} \sim \mathcal{GPW}_{K}^{q}(D, \mathbf{\Sigma}_{K}^{*},\mathbf{I}_{D}, \mathbf{\Omega}, h),
  \quad   i = 1,2,\dots,n,
\end{equation}
where $\mathbf{\Omega} = (\mathbf{\Sigma}_{K}^{*})^{-}\boldsymbol{\mu}^{*}
(\boldsymbol{\mu}^{*})^{T}$.

\begin{remark}
The result in \cite{l:93} is obtained as particular case of (\ref{dBii}), with the difference
that the matrix of noncentrality parameter in \cite{l:93} is defined as $\boldsymbol{\mu}^{*}
(\boldsymbol{\mu}^{*})^{T}$ and we use $\mathbf{\Omega} = (\mathbf{\Sigma}_{K}^{*})^{-}
\boldsymbol{\mu}^{*} (\boldsymbol{\mu}^{*})^{T}$, notation used in \cite[Definition 10.3.1, pp.
441-442]{Mh:82}.
\end{remark}

In addition, defining $\mathbb{X}^{c}$ as $\mathbb{X}$ we have
$$
  \Vec \left( \mathbb{X}^{c}\right)^{T} = [\mathbf{I}_{n} \otimes (\mathbf{H}_{k} \otimes
  \mathbf{I}_{D})] \Vec \mathbb{X}^{T}
$$
hence, $\mathbb{X}^{c} = \mathbb{X} (\mathbf{H}_{k} \otimes \mathbf{I}_{D})$. Now, observing
that $(\mathbf{H}_{k} \otimes \mathbf{I}_{D}) \Vec \mathbf{t}_{i}^{T} = \mathbf{0}$, for all $i
= 1, 2, \dots, n$. Then
\begin{equation}\label{dXXc}
  \mathbb{X}^{c} \sim \mathcal{E}_{n \times KD}^{n,(K-1)D}\left(\sum_{i = 1}^{n} \mathbf{e}_{i}^{n}
  \Vec^{T}(\boldsymbol{\mu}^{*}\mathbf{\Gamma}_{i})^{T},
  \sum_{i = 1}^{n} \mathbf{E}_{ii}^{n} \otimes \mathbf{\Sigma}_{K}^{*}\otimes \mathbf{\Gamma}_{i}^{T}
   \mathbf{\Sigma}_{D}\mathbf{\Gamma}_{i}, h\right),
\end{equation}
where $\mathbf{\Sigma}_{K}^{*} = \mathbf{H}_{k} \mathbf{\Sigma}_{K} \mathbf{H}_{k}$.

As in \cite{l:93}, assuming that $\mathbf{\Sigma}_{D} = \mathbf{I}_{D}$, and recalling that
$$
  \mathbf{I}_{n} = \sum_{i = 1}^{n} \mathbf{E}_{ii}^{n}
$$
we have
$$
  \mathbb{X}^{c} \sim \mathcal{E}_{n \times KD}^{n,(K-1)D}\left(\sum_{i = 1}^{n} \mathbf{e}_{i}^{n}
  \Vec^{T}(\boldsymbol{\mu}^{*}\mathbf{\Gamma}_{i})^{T}, \mathbf{I}_{n} \otimes
  \mathbf{\Sigma}_{K}^{*}\otimes \mathbf{I}_{D}, h\right).
$$

\section{Consistent estimation  of $\boldsymbol{\mu}$ and $\mathbf{\Sigma}_{K}$}\label{sec3}

Alternatively to the use of the Euclidean distance matrix showed in \cite{l:93} with the aim to
propose consistent estimations, we use directly the first two moments of the matrix
$\mathbf{B}$ with the same object.

When is considered a model where the perturbation of landmarks along the $D$ axes are
independent and identical to each other, formally we are assume that $\mathbf{\Sigma}_{D} =
\mathbf{I}_{D}$ under a matrix multivariate Gaussian case. However, this same assumption is not
to hold in matrix multivariate elliptical case. Under a matrix multivariate elliptical case is
possible to consider two cases:
\begin{enumerate}
  \item Independence and not correlation among landmarks and
  \item Probabilistic dependence and not correlation among landmarks.
\end{enumerate}
In both cases $\mathbf{\Sigma}_{D} = \mathbf{I}_{D}$ and the moments of matrix $\mathbf{B}$ are
different in each case.

\begin{remark}
Recall that under matrix multivariate elliptical distribution, only in the Gaussian case the
not correlation and independence are equivalent. Then suppose that the vector $\mathbf{Z} =
(z_{1},z_{2})^{T}$ has a bi-dimensional elliptical distribution and $\cov(\mathbf{Z}) =
\mathbf{I}_{2}$ then $z_{1}$ and $z_{2}$ are independent if and only if $\mathbf{Z}$ has a
bi-dimensional Gaussian distribution. But if $z_{i}$, have a uni-dimensional elliptical
distribution for  $i = 1,2,$ and $\var(z_{i}) = 1$ and $\cov(z_{1},z_{2}) = 0$, $z_{i}$,  $i =
1,2,$ are not correlated and can be considered independent, see \cite[Section 6.2, p. 1]{gv:93}
and \cite[Section 4.3, p. 105]{fkn:90}.
\end{remark}

Summarising, given
\begin{equation}\label{B}
    \mathbf{B} = \mathbf{Y}\mathbf{Y}^{T} = \sum_{d = 1}^{D} \mathbf{y}_{d}
    \mathbf{y}_{d}^{T},
\end{equation}
next, we find the first two moments of $\mathbf{B}$ assuming that $\mathbf{\Sigma}_{D} =
\mathbf{I}_{D}$, i.e. when $\mathbf{y}_{d}$:  a) are not correlated and independent; and b) are
not correlated and dependent.

\subsection{Moments of $\mathbf{B}$ under dependence}

By completeness initially we assume that $\mathbf{\Sigma}_{D} \neq \mathbf{I}_{D}$ and for
convenience denote $\mathbf{\Sigma}_{D} = \mathbf{\Theta}$, $\mathbf{\Sigma}_{K}^{*} =
\mathbf{\Sigma}$ and $\boldsymbol{\mu}^{*} = \boldsymbol{\mu}$.

With this goal in main, suppose that $\mathbf{Y} \sim \mathcal{E}_{K \times
D}^{(K-1),D}(\boldsymbol{\mu}, \mathbf{\Sigma} \otimes \mathbf{\Theta},h)$, with
$$
  \mathbf{Y} = (\mathbf{y}_{1}|\mathbf{y}_{2}| \cdots |\mathbf{y}_{D} ) \mbox{ and }
  \boldsymbol{\mu} = (\boldsymbol{\mu}_{1}|\boldsymbol{\mu}_{2}| \cdots
  |\boldsymbol{\mu}_{D}).
$$
Observing that for $\mathbf{x}, \mathbf{y} \in \Re^{n}$, $\Vec \mathbf{xy}^{T} = \mathbf{y}
\otimes \mathbf{x}$, $\mathbf{xy}^{T} = \mathbf{x}\otimes \mathbf{y}^{T} =
\mathbf{y}^{T}\otimes \mathbf{x}$ and thus, $\Vec \mathbf{y}\mathbf{y}^{T}\Vec^{T}
\mathbf{y}\mathbf{y}^{T} = \mathbf{y} \otimes \mathbf{y}^{T}\otimes \mathbf{y}\otimes
\mathbf{y}^{T}$, see \cite{mn:79}.

\begin{theorem}\label{teo01}
Let $\mathbf{Y} \sim \mathcal{E}_{K \times D}^{(K-1),D}(\boldsymbol{\mu}, \mathbf{\Sigma}
\otimes \mathbf{\Theta},h)$. Then
\begin{enumerate}
  \item $E(\Vec\mathbf{Y}\Vec^{T}\mathbf{Y}) = c_{0} (\mathbf{\Theta} \otimes \mathbf{\Sigma}) +
  \Vec\boldsymbol{\mu}\Vec^{T}\boldsymbol{\mu}$,
  \item and $E(\Vec\mathbf{Y}\Vec\mathbf{Y}^{T} \otimes \Vec\mathbf{Y}\Vec\mathbf{Y}^{T})$ is
  \begin{small}
  \begin{eqnarray*}
     &=&
    \kappa_{0}[(\mathbf{I}_{(KD)^{2}}+\mathbf{K}_{KD})(\mathbf{\Theta} \otimes \mathbf{\Sigma} \otimes
      \mathbf{\Theta} \otimes \mathbf{\Sigma}) + \Vec (\mathbf{\Theta} \otimes \mathbf{\Sigma})
      \Vec^{T}(\mathbf{\Theta} \otimes \mathbf{\Sigma})] \\
    & & +\ c_{0}(\mathbf{I}_{K^{2}}+\mathbf{K}_{K})[\Vec\boldsymbol{\mu}\Vec^{T}\boldsymbol{\mu}
      \otimes (\mathbf{\Theta} \otimes \mathbf{\Sigma}) + (\mathbf{\Theta} \otimes \mathbf{\Sigma}) \otimes
      \Vec\boldsymbol{\mu}\Vec^{T}\boldsymbol{\mu}] \\
    & & +\ c_{0}[\Vec(\mathbf{\Theta} \otimes \mathbf{\Sigma}) (\Vec^{T} \boldsymbol{\mu}\boldsymbol{\mu}^{T}) +
      (\Vec^{T} \boldsymbol{\mu}\boldsymbol{\mu}^{T})\Vec(\mathbf{\Theta} \otimes \mathbf{\Sigma})] \\
    & & +\ \Vec\boldsymbol{\mu}\Vec^{T}\boldsymbol{\mu} \otimes \Vec\boldsymbol{\mu}\Vec^{T}\boldsymbol{\mu},
  \end{eqnarray*}
  \end{small}
\end{enumerate}
where $\mathbf{K}_{KD}$ is the commutation matrix,  see \cite{mn:79}, and  $c_{0} = E(u^{2})$
and $3\kappa_{0} = E(u^{4})$, see \cite[p. 127]{gv:93},
$$
  E(u^{2}) = \left . \frac{1}{i^{2}}\frac{d^{2} \psi_{_{U}}(t)}{dt^{2}}\right|_{t=0} \mbox{ and }
  E(u^{4}) = \left . \frac{1}{i^{4}}\frac{d^{4} \psi_{_{U}}(t)}{dt^{4}}\right|_{t=0}.
$$
Where $\psi_{_{U}}(t) = \phi(t^{2})$ is the characteristic function of univariate elliptical
distribution. Some particular values of $c_{0}$ and $\kappa_{0}$, are summarised on Table
\ref{table1}.
\end{theorem}
\begin{proof}
This is obtained differentiating (\ref{ecf}) and observing that, see \cite{dg:96},
\begin{eqnarray*}
  E(\Vec\mathbf{Y} \otimes \Vec^{T}\mathbf{Y}) &=& E(\Vec\mathbf{Y}\Vec^{T}\mathbf{Y}) \\
   &=& \left . \frac{1}{i^{2}}\frac{\partial^{2}\psi_{\Vec\mathbf{Y}}(\Vec\mathbf{T})}
  {\partial \Vec\mathbf{T}\partial \Vec\mathbf{T}^{T}}\right|_{\Vec\mathbf{T} = 0}
\end{eqnarray*}
and\\%
$ E(\Vec\mathbf{Y} \otimes \Vec\mathbf{Y}^{T} \otimes \Vec\mathbf{Y} \otimes
\Vec\mathbf{Y}^{T}) =  E(\Vec\mathbf{Y}\Vec\mathbf{Y}^{T} \otimes
\Vec\mathbf{Y}\Vec\mathbf{Y}^{T})$
\begin{eqnarray*}
   \phantom{mmmmmmmmmmm}&=& \left . \frac{1}{i^{4}} \frac{\partial^{4}\psi_{\Vec\mathbf{Y}}(\Vec\mathbf{T})}
  {\partial \Vec\mathbf{T}\partial \Vec\mathbf{T}^{T}\partial \Vec\mathbf{T}\partial \Vec\mathbf{T}^{T}}
  \right|_{\Vec\mathbf{T} = 0}.
\end{eqnarray*}
\end{proof}

Now, given
$$
    \mathbf{B} = \mathbf{Y}\mathbf{Y}^{T} = \sum_{d = 1}^{D} \mathbf{y}_{d}
    \mathbf{y}_{d}^{T},
$$
we have
$$
  E(\mathbf{B}) = E \left (\mathbf{Y}\mathbf{Y}^{T}\right) = E \left (\sum_{d = 1}^{D} \mathbf{y}_{d}
    \mathbf{y}_{d}^{T}\right)= \sum_{d = 1}^{D} E \left (\mathbf{y}_{d}
    \mathbf{y}_{d}^{T}\right).
$$

\begin{table}[!h]  \centering \caption{Particular values of $c_{0}$ and $\kappa_{0}$.}\label{table1}
\begin{center}
\begin{minipage}[t]{20cm}
  \begin{tabular}{c|c|c}
   \hline\hline
   Distribution & $c_{0}$ & $\kappa_{0}$\\[1ex]
   \hline\hline
   Multiuniforme\footnote{From \cite[Theorem 3.3, p. 72]{fkn:90}.}& $1$ & $\frac{1}{3}$\\[1ex]
   \hline
   Gaussian\footnote{From \cite[Remark 3.2.2, p. 125]{gv:93}.}& 1 & 1\\[1ex]
   \hline
   Kotz\footnote{From \cite[]{k:03}, where $r,s >0$ and $2N+1 > 2$.} & $\displaystyle \frac{\Gamma\left [ \frac{2N+1}{2s} \right]}
   {r^{1/s} \Gamma\left [\frac{2N-1}{2s}\right ]}$
   & $\displaystyle \frac{\Gamma\left [ \frac{2N+3}{2s} \right]}{3r^{2/s} \Gamma\left [\frac{2N-1}{2s} \right]}$\\[1ex]
   \hline
   $t$\footnote{From \cite[p. 128]{gv:93}, or \cite[p. 88]{fkn:90}, where $m>0$.} &
   $\displaystyle \frac{m}{m-2}$ & $\displaystyle \frac{m^{2}}{(m-2)(m-4)}$\\[1ex]
   \hline
   Pearson Type II\footnote{From \cite[Section 3.4.2, p. 89]{fkn:90}, where $m> -1$.} & $\displaystyle \frac{1}
   {2m+3}$ & $\displaystyle \frac{1}{(2m+3)(2m+5)}$\\[1ex]
   \hline
   Pearson type VII\footnote{From \cite[Section 3.3.4, p. 84]{fkn:90}, where $N >1/2$, $m> 0$.} & $\displaystyle
   \frac{m}{2N-3}$ & $\displaystyle \frac{m^{2}}{(2N-3)(2N-5)}$ \\[1ex]
   \hline\hline
  \end{tabular}
\end{minipage}
\end{center}
\end{table}

\medskip

And remembering that for $\mathbf{Y} \in \Re^{K \times D}$, in general
$$
  \cov (\Vec \mathbf{Y}) = E(\Vec \mathbf{Y} \Vec^{T} \mathbf{Y})-E(\Vec \mathbf{Y})E(\Vec^{T}
  \mathbf{Y}).
$$
Therefore
\begin{eqnarray}
  \cov(\Vec \mathbf{B}) &=& \cov \left (\Vec \left(\mathbf{Y}\mathbf{Y}^{T}\right)\right)
  = \cov \left (\sum_{d = 1}^{D} \Vec\left(\mathbf{y}_{d} \mathbf{y}_{d}^{T}\right)\right)\nonumber\\
  &=& E\left[\left(\sum_{d = 1}^{D} \Vec\left (\mathbf{y}_{d}\mathbf{y}_{d}^{T}\right)\right)
  \left(\sum_{s = 1}^{D} \Vec^{T} \left (\mathbf{y}_{s}\mathbf{y}_{s}^{T}\right)\right)\right]\nonumber\\
  && - \ E\left(\sum_{d = 1}^{D} \Vec\left (\mathbf{y}_{d}\mathbf{y}_{d}^{T}\right)\right)
  E\left(\sum_{d = 1}^{D} \Vec^{T}\left (\mathbf{y}_{d}\mathbf{y}_{d}^{T}\right)\right)\nonumber\\
  &=& \left[\sum_{d = 1}^{D}\sum_{s = 1}^{D} E\left (\mathbf{y}_{d}\mathbf{y}_{s}^{T}
  \otimes \mathbf{y}_{d}\mathbf{y}_{s}^{T}\right)\right]\nonumber\\\label{mc01}
  && - \ \Vec\left(\sum_{d = 1}^{D} E\left (\mathbf{y}_{d}\mathbf{y}_{d}^{T}\right)\right)
  \Vec^{T}\left(\sum_{s = 1}^{D} E\left (\mathbf{y}_{s}\mathbf{y}_{s}^{T}\right)\right)
\end{eqnarray}
Then, we need to find $ E \left (\mathbf{y}_{d} \mathbf{y}_{d}^{T}\right)$ and $E\left
(\mathbf{y}_{d}\mathbf{y}_{s}^{T} \otimes \mathbf{y}_{d}\mathbf{y}_{s}^{T}\right)$. These
moments are obtained in the following result.

\begin{theorem}\label{teo02}
Assume that $\mathbf{Y} \sim \mathcal{E}_{K \times D}^{(K-1),D}(\boldsymbol{\mu},
\mathbf{\Sigma} \otimes \mathbf{\Theta},h)$, with
$$
  \mathbf{Y} = (\mathbf{y}_{1}|\mathbf{y}_{2}| \cdots |\mathbf{y}_{D} ) \mbox{ and }
  \boldsymbol{\mu} = (\boldsymbol{\mu}_{1}|\boldsymbol{\mu}_{2}| \cdots
  |\boldsymbol{\mu}_{D}),
$$
and $\mathbf{\Theta} = (\theta_{ds})$. Then
\begin{enumerate}
  \item $E(\mathbf{y}_{d}\mathbf{y}^{T}_{d}) = c_{0}\theta_{dd}\mathbf{\Sigma} + \boldsymbol{\mu}_{d}
  \boldsymbol{\mu}^{T}_{d}$.
  \item And
\begin{eqnarray*}
    E\left(\mathbf{y}_{d}\mathbf{y}_{s}^{T} \otimes \mathbf{y}_{d}\mathbf{y}_{s}^{T}\right)
   &=& \kappa_{0}\theta_{ds}^{2}[(\mathbf{I}_{K^{2}}+\mathbf{K}_{K})(\mathbf{\Sigma}\otimes
      \mathbf{\Sigma}) + \Vec \mathbf{\Sigma} \Vec^{T}\mathbf{\Sigma}]\\
   && + \ c_{0}\theta_{ds}[(\mathbf{I}_{K^{2}}+\mathbf{K}_{K})(\boldsymbol{\mu}_{d}\boldsymbol{\mu}^{T}_{s}
      \otimes \mathbf{\Sigma} + \mathbf{\Sigma} \otimes  \boldsymbol{\mu}_{d}\boldsymbol{\mu}^{T}_{s})]\\
   && + \ c_{0}\theta_{ds}[\Vec \mathbf{\Sigma} \Vec^{T} \boldsymbol{\mu}_{d}\boldsymbol{\mu}^{T}_{s}
      +\Vec \boldsymbol{\mu}_{d}\boldsymbol{\mu}^{T}_{s}\Vec^{T}\mathbf{\Sigma}]\\
   && + \boldsymbol{\mu}_{d}\boldsymbol{\mu}^{T}_{s} \otimes \boldsymbol{\mu}_{d}\boldsymbol{\mu}^{T}_{s}.
\end{eqnarray*}
\end{enumerate}
\end{theorem}
\begin{proof}
The results is obtained as consequence of Theorem \ref{teo01} observing that $\mathbf{y}_{d} =
\mathbf{Ye}_{d}^{D}$, then
\begin{eqnarray*}
  E(\mathbf{y}_{d}\mathbf{y}^{T}_{d}) &=& E(\Vec\mathbf{y}_{d} \Vec^{T}\mathbf{y}_{d}) =
  E(\Vec\mathbf{Ye}_{d}^{D} \Vec^{T}\mathbf{Ye}_{d}^{D}) \\
   &=& (\mathbf{e}_{d}^{D \ T} \otimes \mathbf{I}_{K})E(\Vec\mathbf{Y} \Vec^{T}\mathbf{Y})
   (\mathbf{e}_{d}^{D} \otimes \mathbf{I}_{K})\\
   &=& (\mathbf{e}_{d}^{D \ T} \otimes \mathbf{I}_{K})(c_{0} (\mathbf{\Theta} \otimes \mathbf{\Sigma}) +
  \Vec\boldsymbol{\mu}\Vec^{T}\boldsymbol{\mu}) (\mathbf{e}_{d}^{D} \otimes \mathbf{I}_{K})\\
  &=& c_{0}\theta_{dd}\mathbf{\Sigma} + \boldsymbol{\mu}_{d}\boldsymbol{\mu}^{T}_{d}.
\end{eqnarray*}
This least result is obtained noting that, $\Vec \mathbf{ABC} = (\mathbf{C}^{T} \otimes
\mathbf{B})\Vec \mathbf{B}$, $a \otimes \mathbf{A} = a\mathbf{A}$ and $(\mathbf{A} \otimes
\mathbf{D})(\mathbf{B} \otimes \mathbf{E})(\mathbf{C} \otimes \mathbf{F}) = (\mathbf{ABC}
\otimes \mathbf{DEF})$. Similarly,
$$
  E\left(\mathbf{y}_{d}\mathbf{y}_{s}^{T} \otimes \mathbf{y}_{d}\mathbf{y}_{s}^{T}\right)
  = \mathbf{R}^{T} E(\Vec\mathbf{Y}\Vec\mathbf{Y}^{T} \otimes \Vec\mathbf{Y}\Vec\mathbf{Y}^{T})\mathbf{R}_{1}\\
$$
with $\mathbf{R}^{T} = (\mathbf{e}_{d}^{D \ T} \otimes \mathbf{I}_{K}) \otimes (\mathbf{e}_{d}^{D \ T}
\otimes \mathbf{I}_{K})$ and $\mathbf{R}_{1} = (\mathbf{e}_{s}^{D } \otimes \mathbf{I}_{K}) \otimes
(\mathbf{e}_{s}^{D } \otimes \mathbf{I}_{K})$. The desired result is obtained observing that: for
$\mathbf{A} \in \Re^{n \times s}$ and $\mathbf{B} \in \Re^{m \times t}$, $\mathbf{K}_{mn}(\mathbf{A}
\otimes \mathbf{B}) = (\mathbf{B} \otimes \mathbf{A})\mathbf{K}_{ts}$ and that $\mathbf{K}_{mm} \equiv
\mathbf{K}_{m}$ see \cite{mn:79}.
\end{proof}

Consider the following definition.
\begin{definition}\label{def02}
Let $\mathbf{A} \in \Re^{p \times q}$ such that
$$
  \mathbf{A} =
  \left (
  \begin{array}{cccc}
    \mathbf{A}_{11} & \mathbf{A}_{12} & \cdots & \mathbf{A}_{1n} \\
    \mathbf{A}_{21} & \mathbf{A}_{22} & \cdots & \mathbf{A}_{2n} \\
    \vdots & \vdots & \ddots & \vdots \\
    \mathbf{A}_{m1} & \mathbf{A}_{m2} & \cdots & \mathbf{A}_{mm}
  \end{array}
  \right ), \quad \mathbf{A}_{ij}\in \Re^{r \times s}
$$
with, $mr = p$ and $ns = q$, then
$$
  \build{\boxplus}{m,n}{i,j} \mathbf{A} = \sum_{i = 1}^{m} \sum_{j = 1}^{n} \mathbf{A}_{ij} \in \Re^{r \times s}.
$$
If $m = n$ then, $\build{\boxplus}{m,m}{i,j} \equiv \build{\boxplus}{m}{i,j}$.
\end{definition}

In addition, let $\mathbf{A} = (\mathbf{A}_{ij})$ and $\mathbf{B} = (\mathbf{B}_{ij})$
partitioned matrices. Then if $\odot$ denotes the Khatri-Rao product, see \cite[p.30]{r:73},
$$
  \mathbf{A} \odot \mathbf{B} = \left(\mathbf{A}_{ij} \otimes \mathbf{B}_{ij} \right)_{ij}.
$$
In particular, note that if $\mathbf{C} = (c_{ij})$, then
$$
  \mathbf{C} \odot \mathbf{A} = \left(c_{ij}\mathbf{A}_{ij}\right)_{ij}.
$$
Moreover,
$$
  \build{\boxplus}{}{i,j} (\mathbf{C} \odot \mathbf{A}) = \sum_{i} \sum_{j} \left(c_{ij}\mathbf{A}_{ij}\right)_{ij}.
$$

\begin{theorem}\label{teo03}
Suppose that $\mathbf{Y} \sim \mathcal{E}_{K \times D}^{(K-1),D}(\boldsymbol{\mu},
\mathbf{\Sigma} \otimes \mathbf{\Theta},h)$, with
$$
  \mathbf{Y} = (\mathbf{y}_{1}|\mathbf{y}_{2}| \cdots |\mathbf{y}_{D} ) \mbox{ and }
  \boldsymbol{\mu} = (\boldsymbol{\mu}_{1}|\boldsymbol{\mu}_{2}| \cdots
  |\boldsymbol{\mu}_{D}).
$$
And define
$$
    \mathbf{B} = \mathbf{Y}\mathbf{Y}^{T} = \sum_{d = 1}^{D} \mathbf{y}_{d}
    \mathbf{y}_{d}^{T}.
$$
Then
$$
  E(\mathbf{B}) = c_{0} \tr(\mathbf{\Theta})\mathbf{\Sigma} + \boldsymbol{\mu}
  \boldsymbol{\mu}^{T}.
$$
And
\begin{eqnarray*}
    \cov(\Vec \mathbf{B}))
   &=& \left(\mathbf{I}_{K^{2}}+\mathbf{K}_{K}\right)\left\{\kappa_{0}\tr\left(\mathbf{\Theta}^{2} \right)(\mathbf{\Sigma} \otimes
      \mathbf{\Sigma})  \right .\\
   && + \ c_{0}\left[\build{\boxplus}{D}{i,j}\left(\mathbf{\Theta} \odot \Vec \boldsymbol{\mu}\Vec^{T}\boldsymbol{\mu}\right)
      \otimes \mathbf{\Sigma} \right.\\
   && \left . \left . +  \mathbf{\Sigma} \otimes  \build{\boxplus}{D}{i,j}\left(\mathbf{\Theta} \odot \Vec \boldsymbol{\mu}
      \Vec^{T} \boldsymbol{\mu}\right)\right]\right\}\\
   && + \ \left[\kappa_{0}\tr\left(\mathbf{\Theta}^{2}\right) - c_{0}^{2} \tr^{2}(\mathbf{\Theta})\right] \Vec \mathbf{\Sigma}
   \Vec^{T}\mathbf{\Sigma}\\
   && + \ c_{0}\left\{\Vec \mathbf{\Sigma} \Vec^{T} \build{\boxplus}{D}{i,j}\left(\mathbf{\Theta} \odot \Vec
   \boldsymbol{\mu} \Vec^{T}\boldsymbol{\mu}\right)\right.\\
   && + \ \Vec\build{\boxplus}{D}{i,j}\left(\mathbf{\Theta} \odot \Vec \boldsymbol{\mu}\Vec^{T}\boldsymbol{\mu}\right) \Vec^{T}\mathbf{\Sigma}\\
   && \left . + \tr(\mathbf{\Theta})\left[\Vec \mathbf{\Sigma} \Vec^{T} \boldsymbol{\mu}\boldsymbol{\mu}^{T}
   + \Vec \boldsymbol{\mu}\boldsymbol{\mu}^{T} \Vec \mathbf{\Sigma}\right] \right\}.
\end{eqnarray*}
\end{theorem}
\begin{proof}
This is a consequence of (\ref{B}), (\ref{mc01}), Definition \ref{def02} and Theorem 2.
\end{proof}

\begin{corollary}
In Theorem \ref{teo03} assume that $\mathbf{\Theta} = \mathbf{I}_{D}$. Then
$$
  E(\mathbf{B}) = D c_{0} \mathbf{\Sigma} + \boldsymbol{\mu}
  \boldsymbol{\mu}^{T}.
$$
And
\begin{eqnarray*}
    \cov(\Vec \mathbf{B}))
   &=& \left(\mathbf{I}_{K^{2}}+\mathbf{K}_{K}\right)\left\{ D \kappa_{0}(\mathbf{\Sigma} \otimes
      \mathbf{\Sigma})+ c_{0}\left[\boldsymbol{\mu}\boldsymbol{\mu}^{T} \otimes \mathbf{\Sigma} +
      \mathbf{\Sigma} \otimes \boldsymbol{\mu} \boldsymbol{\mu}^{T}\right]\right\}\\
   && + \ D\left[\kappa_{0} - D c_{0}^{2} \right] \Vec \mathbf{\Sigma}
   \Vec^{T}\mathbf{\Sigma}\\
   && + \ (1-D)c_{0}[\Vec \mathbf{\Sigma} \Vec^{T}\boldsymbol{\mu}\boldsymbol{\mu}^{T}
   + \Vec \boldsymbol{\mu}\boldsymbol{\mu}^{T}\Vec^{T} \mathbf{\Sigma}].
\end{eqnarray*}
\end{corollary}

In univariate case, when $\boldsymbol{\mu} = \mathbf{0}$, these results were obtained in
general and a particular cases in \cite[Theorem 3.2.13 and Example 3.2.1]{gv:93},  with a
several minor errors. In particular, for general case they write $D^{2}(\kappa_{0}-c_{0}^{2})$
and for matrix multivariate $T$ distribution they write $D(\kappa_{0}-c_{0}^{2})$, with $D =
n-1$, instead of $D\left(\kappa_{0} - D c_{0}^{2} \right)$.

\subsection{Moments of $\mathbf{B}$ under independence}

Let $\mathbf{Y}$ and $\boldsymbol{\mu}$ such that
$$
  \mathbf{Y} = (\mathbf{y}_{1}|\mathbf{y}_{2}| \cdots |\mathbf{y}_{D} ) \mbox{ and }
  \boldsymbol{\mu} = (\boldsymbol{\mu}_{1}|\boldsymbol{\mu}_{2}| \cdots
  |\boldsymbol{\mu}_{D}),
$$
where $\mathbf{y}_{1},\mathbf{y}_{2}, \dots, \mathbf{y}_{D}$ are independent and
$$
   \mathbf{y}_{d} \sim  \mathcal{E}_{K}^{(K-1)}(\boldsymbol{\mu}_{d}, \theta_{dd}\mathbf{\Sigma}; h),
$$
and by independence, $\cov(\mathbf{y}_{d}, \mathbf{y}_{s}) = \mathbf{0}$, for $d \neq s =
1,2\dots,D$.

Given
$$
    \mathbf{B} = \mathbf{Y}\mathbf{Y}^{T} = \sum_{d = 1}^{D} \mathbf{y}_{d}
    \mathbf{y}_{d}^{T},
$$
we have
$$
  E(\mathbf{B}) = E \left (\mathbf{Y}\mathbf{Y}^{T}\right) = E \left (\sum_{d = 1}^{D} \mathbf{y}_{d}
    \mathbf{y}_{d}^{T}\right)= \sum_{d = 1}^{D} E \left (\mathbf{y}_{d}
    \mathbf{y}_{d}^{T}\right).
$$

And under assumption that $\mathbf{y}_{d}$, $d = 1,2 ,\dots,D$ are independent,
\begin{eqnarray*}
  \cov(\Vec \mathbf{B}) &=& \cov \left (\Vec \left(\mathbf{Y}\mathbf{Y}^{T}\right)\right)
  = \cov \left (\sum_{d = 1}^{D} \Vec\left(\mathbf{y}_{d} \mathbf{y}_{d}^{T}\right)\right)\\
  &=& \sum_{d = 1}^{D} \cov\left(\Vec\left (\mathbf{y}_{d}\mathbf{y}_{d}^{T}\right)\right)
  = \sum_{d = 1}^{D} \cov\left(\mathbf{y}_{d}\otimes \mathbf{y}_{d}\right).
\end{eqnarray*}
Then, we need to find $ E \left (\mathbf{y}_{d} \mathbf{y}_{d}^{T}\right)$ and
\begin{eqnarray*}
  \cov\left(\mathbf{y}_{d}\otimes \mathbf{y}_{d}\right) &=& E(\left(\mathbf{y}_{d}\otimes \mathbf{y}_{d}\right)
  \left(\mathbf{y}_{d}\otimes \mathbf{y}_{d}\right)^{T})-E\left(\mathbf{y}_{d}\otimes
  \mathbf{y}_{d}\right)E\left(\mathbf{y}_{d}\otimes \mathbf{y}_{d}\right)^{T} \\
   &=& E\left(\mathbf{y}_{d}\mathbf{y}_{d}^{T}\otimes \mathbf{y}_{d}\mathbf{y}_{d}^{T}\right)
  -E\left(\Vec \mathbf{y}_{d}\mathbf{y}_{d}^{T}\right)E\left(\Vec^{T}
  \mathbf{y}_{d}\mathbf{y}_{d}^{T}\right).
\end{eqnarray*}

These results are obtained in the following

\begin{corollary}\label{cor00}
Let $\mathbf{y}_{d} \sim  \mathcal{E}_{K}^{(K-1)}(\boldsymbol{\mu}_{d},
\theta_{dd}\mathbf{\Sigma}; h), \quad d = 1, 2, \dots, D$, where
$\mathbf{y}_{1},\mathbf{y}_{2}, \dots, \mathbf{y}_{D}$ are independent. Then
\begin{enumerate}
  \item $E(\mathbf{y}_{d}\mathbf{y}^{T}_{d}) = c_{0}\theta_{dd}\mathbf{\Sigma} + \boldsymbol{\mu}_{d}
  \boldsymbol{\mu}^{T}_{d}$.
  \item And $\cov(\mathbf{y}_{d} \otimes \mathbf{y}_{d}) = \cov(\Vec\mathbf{y}_{d}\mathbf{y}^{T}_{d})$
\begin{eqnarray*}
       &=& (\mathbf{I}_{K^{2}}+\mathbf{K}_{K})\left\{\kappa_{0} \theta_{dd}^{2} (\mathbf{\Sigma} \otimes \mathbf{\Sigma}) +
      c_{0}\theta_{dd}\left[\boldsymbol{\mu}_{d}\boldsymbol{\mu}^{T}_{d}  \otimes \mathbf{\Sigma} + \mathbf{\Sigma} \otimes
      \boldsymbol{\mu}_{d}\boldsymbol{\mu}^{T}_{d}\right]\right \} \\
   && + \ \theta_{dd}^{2}(\kappa_{0} - c^{2}_{0})\Vec \mathbf{\Sigma}   \Vec^{T}\mathbf{\Sigma}.
\end{eqnarray*}
\end{enumerate}
\end{corollary}
\begin{proof}
It is follows from Theorem \ref{teo02}, taking $d = s$.
\end{proof}

\begin{theorem}\label{teo04}
Suppose that $\mathbf{y}_{d} \sim  \mathcal{E}_{K}^{(K-1)}(\boldsymbol{\mu}_{d},
\theta_{dd}\mathbf{\Sigma}; h), \quad d = 1, 2, \dots,D$, with
$$
  \mathbf{Y} = (\mathbf{y}_{1}|\mathbf{y}_{2}| \cdots |\mathbf{y}_{D} ) \mbox{ and }
  \boldsymbol{\mu} = (\boldsymbol{\mu}_{1}|\boldsymbol{\mu}_{2}| \cdots |\boldsymbol{\mu}_{D}
  ),
$$
and let $\mathbf{\Theta} = \diag(\theta_{11}, \theta_{22}, \dots,\theta_{dd})$, and
$$
    \mathbf{B} = \mathbf{Y}\mathbf{Y}^{T} = \sum_{d = 1}^{D} \mathbf{y}_{d}
    \mathbf{y}_{d}^{T}.
$$
Then,
\begin{eqnarray*}
  E(\mathbf{B}) &=& c_{0} \tr(\mathbf{\Theta}) \mathbf{\Sigma} + \boldsymbol{\mu}\boldsymbol{\mu}^{T}\\
  \cov(\Vec \mathbf{B}) &=& \left(\mathbf{I}_{K^{2}}+\mathbf{K}_{K}\right)\left\{\kappa_{0}
    \tr\left(\mathbf{\Theta}^{2}\right) (\mathbf{\Sigma} \otimes \mathbf{\Sigma})\right .\\
    & &\left . + \ c_{0}\left[\left(\sum_{d=1}^{D}\theta_{dd}\boldsymbol{\mu}_{d}\boldsymbol{\mu}^{T}_{d}\right )
    \otimes \mathbf{\Sigma} + \mathbf{\Sigma} \otimes
    \left(\sum_{d = 1}^{D}\theta_{dd}\boldsymbol{\mu}_{d}\boldsymbol{\mu}^{T}_{d}\right)\right]\right\}\\
    && + \ (\kappa_{0} - c^{2}_{0})\tr\left(\mathbf{\Theta}^{2}\right)\Vec \mathbf{\Sigma} \Vec^{T}\mathbf{\Sigma}
\end{eqnarray*}
\end{theorem}
\begin{proof}
From Corollary \ref{cor00},
\begin{eqnarray*}
  E(\mathbf{B}) &=& E \left (\mathbf{Y}\mathbf{Y}^{T}\right) = E \left (\sum_{d = 1}^{D} \mathbf{y}_{d}
    \mathbf{y}_{d}^{T}\right)= \sum_{d = 1}^{D} E \left (\mathbf{y}_{d}
    \mathbf{y}_{d}^{T}\right)\\
    &=& \sum_{d = 1}^{D}\left(c_{0} \theta_{dd}\mathbf{\Sigma} +
    \boldsymbol{\mu}_{d}\boldsymbol{\mu}^{T}_{d}\right) \\
    &=& c_{0} \tr(\mathbf{\Theta})\mathbf{\Sigma} +
    \sum_{d = 1}^{D}\boldsymbol{\mu}_{d}\boldsymbol{\mu}^{T}_{d} =  c_{0} \tr(\mathbf{\Theta})\mathbf{\Sigma} +
    \boldsymbol{\mu}\boldsymbol{\mu}^{T}.
\end{eqnarray*}
Similarly,
\begin{eqnarray*}
  \cov(\Vec \mathbf{B}) &=& \cov \left (\Vec \left(\mathbf{Y}\mathbf{Y}^{T}\right)\right)
  = \cov \left (\sum_{d = 1}^{D} \Vec\left(\mathbf{y}_{d} \mathbf{y}_{d}^{T}\right)\right)\\
  &=& \sum_{d = 1}^{D} \cov\left(\Vec\left (\mathbf{y}_{d}\mathbf{y}_{d}^{T}\right)\right).
\end{eqnarray*}
from the desired result is obtained.
\end{proof}

Now if $\mathbf{\Theta} = \mathbf{I}_{D}$, we have the following results.

\begin{corollary}\label{cor01}
Let $\mathbf{y}_{d} \sim  \mathcal{E}_{K}^{(K-1)}(\boldsymbol{\mu}_{d}, \mathbf{\Sigma}; h),
\quad d = 1, 2, \dots, D$, where $\mathbf{y}_{1},\mathbf{y}_{2}, \dots, \mathbf{y}_{D}$ are
independent. Then
\begin{enumerate}
  \item $E(\mathbf{y}_{d}\mathbf{y}^{T}_{d}) = c_{0}\mathbf{\Sigma} + \boldsymbol{\mu}_{d}
  \boldsymbol{\mu}^{T}_{d}$.
  \item and $\cov(\mathbf{y}_{d} \otimes \mathbf{y}_{d}) = \cov(\Vec\mathbf{y}_{d}\mathbf{y}^{T}_{d})$
  \begin{eqnarray*}
   &=& (\mathbf{I}_{K^{2}}+\mathbf{K}_{K})[\kappa_{0} (\mathbf{\Sigma} \otimes \mathbf{\Sigma}) +
      c_{0}(\boldsymbol{\mu}_{d}\boldsymbol{\mu}^{T}_{d}  \otimes \mathbf{\Sigma} + \mathbf{\Sigma} \otimes
      \boldsymbol{\mu}_{d}\boldsymbol{\mu}^{T}_{d})] \\
   && + (\kappa_{0} - c^{2}_{0})\Vec \mathbf{\Sigma}   \Vec^{T}\mathbf{\Sigma},
\end{eqnarray*}
\end{enumerate}
\end{corollary}
\begin{proof}
It is immediately.
\end{proof}
\begin{theorem}
Suppose that $\mathbf{y}_{d} \sim  \mathcal{E}_{K}^{(K-1)}(\boldsymbol{\mu}_{d},
\mathbf{\Sigma}; h), \quad d = 1, 2, \dots,D$, with
$$
  \mathbf{Y} = (\mathbf{y}_{1}|\mathbf{y}_{2}| \cdots |\mathbf{y}_{D} ) \mbox{ and }
  \boldsymbol{\mu} = (\boldsymbol{\mu}_{1}|\boldsymbol{\mu}_{2}| \cdots |\boldsymbol{\mu}_{D}
  ),
$$
and let
$$
    \mathbf{B} = \mathbf{Y}\mathbf{Y}^{T} = \sum_{d = 1}^{D} \mathbf{y}_{d}
    \mathbf{y}_{d}^{T}.
$$
Then,
\begin{eqnarray*}
  E(\mathbf{B}) &=& D c_{0} \mathbf{\Sigma} + \boldsymbol{\mu}\boldsymbol{\mu}^{T}\\
  \cov(\Vec \mathbf{B}) &=& (\mathbf{I}_{K^{2}}+\mathbf{K}_{K})[D \kappa_{0} (\mathbf{\Sigma} \otimes \mathbf{\Sigma}) +
  c_{0}(\boldsymbol{\mu}\boldsymbol{\mu}^{T}  \otimes \mathbf{\Sigma} + \mathbf{\Sigma} \otimes
      \boldsymbol{\mu}\boldsymbol{\mu}^{T})]\\
      && + \ D(\kappa_{0} - c^{2}_{0})\Vec \mathbf{\Sigma} \Vec^{T}\mathbf{\Sigma}
\end{eqnarray*}
\end{theorem}
\begin{proof}
This is obtained from Theorem \ref{teo04}.
\end{proof}

\begin{corollary}
In particular if $\mathbf{Y} \sim \mathcal{N}_{K \times D}^{(K-1),D}(\boldsymbol{\mu},
\mathbf{\Sigma} \otimes \mathbf{I}_{D})$. Then, $c_{0} = \kappa_{0} = 1$, and thus
\begin{eqnarray*}
  E(\mathbf{B}) &=& D \mathbf{\Sigma} + \boldsymbol{\mu}\boldsymbol{\mu}^{T}\\
  \cov(\Vec \mathbf{B}) &=& (\mathbf{I}_{K^{2}}+\mathbf{K}_{K})[D (\mathbf{\Sigma} \otimes \mathbf{\Sigma}) +
  \boldsymbol{\mu}\boldsymbol{\mu}^{T}_{d}  \otimes \mathbf{\Sigma} + \mathbf{\Sigma} \otimes
      \boldsymbol{\mu}\boldsymbol{\mu}^{T}].
\end{eqnarray*}
\end{corollary}

\subsection{Method-of-moments estimators}

Returning to our notation, for which, rewrite, $\mathbf{\Theta} = \mathbf{\Sigma}_{D}$,
$\mathbf{\Sigma} = \mathbf{\Sigma}_{K}^{*}$ and $\boldsymbol{\mu} = \boldsymbol{\mu}^{*}$.

Our target is to find the method-of-moments estimators of the parameter matrices
$\mathbf{\Sigma}_{K}^{*}$ and $\boldsymbol{\mu}^{*}$. First, note that the first two sample
moments estimators of $\mathbf{B}$ are given by
$$
  \widetilde{E(\mathbf{B})}=  \frac{1}{n}\sum_{i = 1}^{n} \mathbf{B}_{i} = \bar{\mathbf{B}} =
  (\bar{b}_{ij}), \quad i,j = 1,\dots,K,
$$
and
$$
  \widetilde{\cov(\Vec \mathbf{B})} = \frac{1}{n}\sum_{i = 1}^{n} (\Vec \mathbf{B}_{i}^{T}-
  \Vec \widetilde{E(\mathbf{B})})(\Vec \mathbf{B}_{i}^{T}- \Vec \widetilde{E(\mathbf{B})})^{T} = \mathbf{S}.
$$
where $\mathbf{S} = (s_{_{tr}})$, $t,r = 1,2,\dots,K^{2}$. In addition note that for $i \leq j$
and $\mathbf{M} = \boldsymbol{\mu}^{*} \boldsymbol{\mu}^{*T} = (m_{ij}) = \mathbf{M}^{T}$ and
$\mathbf{\Sigma}_{K}^{*} = (\sigma_{ij})$, we have
\begin{eqnarray}
  E(b_{ij}) &=& E(\mathbf{e}_{i}^{T}\mathbf{B}\mathbf{e}_{j}) = \mathbf{e}_{i}^{T}E(\mathbf{B})\mathbf{e}_{j} \nonumber\\
  \label{egwd}
   &=& \mathbf{e}_{i}^{T}(D c_{0}\mathbf{\Sigma}_{K}^{*} +
    \boldsymbol{\mu}^{*}\boldsymbol{\mu}^{*T})\mathbf{e}_{j} = D c_{0}\sigma_{ij}+ m_{ij},
\end{eqnarray}
for independent and dependent cases.

\subsubsection{Dependent case}

Note that:
\begin{eqnarray*}
  \cov(b_{ij}) &=& \cov(\mathbf{e}_{i}^{T}\mathbf{B}\mathbf{e}_{j}) = \cov(\Vec\mathbf{e}_{i}^{T}\mathbf{B}\mathbf{e}_{j})
   = \cov((\mathbf{e}_{j} \otimes \mathbf{e}_{i})^{T}\Vec \mathbf{B}) \\
   &=& (\mathbf{e}_{j} \otimes \mathbf{e}_{i})^{T}\cov(\Vec \mathbf{B}) (\mathbf{e}_{j} \otimes \mathbf{e}_{i})\\
   &=& (\mathbf{e}_{j} \otimes \mathbf{e}_{i})^{T}\{\left(\mathbf{I}_{K^{2}}+\mathbf{K}_{K}\right)
   \left\{ D \kappa_{0}(\mathbf{\Sigma}_{K}^{*} \otimes \mathbf{\Sigma}_{K}^{*}) \right. \\
   & & \left . + \ c_{0}\left[\boldsymbol{\mu}^{*}\boldsymbol{\mu}^{* T} \otimes \mathbf{\Sigma}_{K}^{*} +
      \mathbf{\Sigma}_{K}^{*} \otimes \boldsymbol{\mu}^{*} \boldsymbol{\mu}^{* T}\right]\right\}\\
   && + \ D\left[\kappa_{0} - D c_{0}^{2} \right] \Vec \mathbf{\Sigma}_{K}^{*}
   \Vec^{T}\mathbf{\Sigma}_{K}^{*}\\
   &&  + \ (1-D)c_{0}[\Vec \mathbf{\Sigma}_{K}^{*} \Vec^{T}\boldsymbol{\mu}^{*}\boldsymbol{\mu}^{T* }\\
   && \left . + \ \Vec \boldsymbol{\mu}^{*}\boldsymbol{\mu}^{* T}\Vec^{T} \mathbf{\Sigma}_{K}^{*}]\right\}(\mathbf{e}_{j} \otimes
      \mathbf{e}_{i}).
\end{eqnarray*}
Observing that $(\mathbf{e}_{j} \otimes \mathbf{e}_{i})^{T}\mathbf{K}_{K} = (\mathbf{e}_{i}
\otimes \mathbf{e}_{j})^{T}$ and that $(\mathbf{A} \otimes \mathbf{B})(\mathbf{C} \otimes
\mathbf{D}) = (\mathbf{AC} \otimes \mathbf{BD})$, for $i \leq j$, $i, j = 1,2,\dots,K$
\begin{eqnarray}
   \cov(b_{ij}) &=& D\left[\kappa_{0}\sigma_{ii}\sigma_{jj} + (2\kappa_{0} - c_{0}^{2})\sigma_{ij}^{2}\right] \nonumber\\
   \label{vgwd}
   && + \ c_{0} \left[m_{jj}\sigma_{ii} + m_{ii}\sigma_{jj}+ 2(2-D)m_{ij}\sigma_{ij}\right].
\end{eqnarray}
From (\ref{egwd}),  by replacing $m_{ij} = \bar{b}_{ij}-Dc_{0}
\sigma_{ij}$  in (\ref{vgwd}) we have
\begin{eqnarray}
  \cov(b_{ij}) &=& D (\kappa_{0}-2c_{0}^{2}) \sigma_{ii}\sigma_{jj} + D(2\kappa_{0} - (1+ 2(2-D)) c_{0}^{2})\sigma_{ij}^{2} \nonumber\\
  \label{sem11}
   && + \ c_{0} \left[\bar{b}_{jj}\sigma_{ii}+ \bar{b}_{ii} \sigma_{jj} + 2(2-D)\bar{b}_{ij}\sigma_{ij}\right].
\end{eqnarray}
Therefore equaling (\ref{sem11}) to $s_{ij} =\widetilde{
\cov(b_{ij})}$ we have:
\begin{theorem}\label{th:6}
Assume that $\mathbf{B}\sim \mathcal{GPW}_{K}^{q}(D, \mathbf{\Sigma}_{K}^{*}, \mathbf{I}_{D},
\mathbf{\Omega},h)$. Then, the method-of-moments estimators of $\mathbf{\Sigma}_{K}^{*}$ and
$\mathbf{M}=\boldsymbol{\mu}^{*}\boldsymbol{\mu}^{*T}$ are given by the following exact
expressions.

\textbf{For $i=1,2,\ldots,K$:}
\begin{equation}\label{sigmaii}
    \widetilde{\sigma}_{ii}=\frac{\sqrt{Q_{ii}^{2}+4Ps_{ii}}-Q_{ii}}{2P},
\end{equation}
where $Q_{ii}^{2}+4Ps_{ii}\geq 0$,
$P=D(\kappa_{0}-2c_{0}^{2})+D(2\kappa_{0}-(1+2(2-D))c_{0}^{2})$, and
$Q_{ii}=2c_{0}(3-D)\overline{b}_{ii}$.
\begin{equation}\label{mii}
    \widetilde{m}_{ii}=\overline{b}_{ii}-Dc_{0}\widetilde{\sigma}_{ii},
\end{equation}
where $\widetilde{\sigma}_{ii}$ has been previously found in (\ref{sigmaii}).

If $P=0$,  then $\widetilde{\sigma}_{ii}=s_{ii}/Q_{ii}$.

\textbf{For $i<j$, $i=1,\ldots,(K-1),j=2,\ldots,K$:}
\begin{equation}\label{sigmaij}
    \widetilde{\sigma}_{ij}=\frac{\sqrt{(2-D)^{2}c_{0}^{2}\overline{b}_{ij}^{2}-R(T_{ij}-s_{ij})}-(2-D)c_{0}\overline{b}_{ij}}{R},
\end{equation}
where $(2-D)^{2}c_{0}^{2}\overline{b}_{ij}^{2}-R(T_{ij}-s_{ij})\geq
0$, $R=D(2\kappa_{0}-(1+2(2-D))c_{0}^{2})$, and
$$
 T_{ij}=D(\kappa_{0}-2c_{0}^{2})\widetilde{\sigma}_{ii}\widetilde{\sigma}_{jj}+c_{0}(\overline{b}_{jj}\widetilde{\sigma}_{ii}
+ \overline{b}_{ii}\widetilde{\sigma}_{jj}).
$$
Here $\widetilde{\sigma}_{ii}$ and $\widetilde{\sigma}_{jj}$ were previously computed in
(\ref{sigmaii}).
\begin{equation}\label{mii2}
    \widetilde{m}_{ij}=\overline{b}_{ij}-Dc_{0}\widetilde{\sigma}_{ij},
\end{equation}

Denote the solution as
$$
(\widetilde{\mathbf{M}},\widetilde{\mathbf{\Sigma}}_{K}^{*}).
$$
Note that $s_{ij}=\widetilde{Cov(b_{ij})}$, this is $s_{ij}$ are obtained from the diagonal of
matrix $\mathbf{S}\in \mathfrak{R}^{K^{2}\times K^{2}}$.

If $R=0$, then
$\widetilde{\sigma}_{ij}=\left(s_{ij}-T_{ij}\right)/\left(2(2-D)c_{0}\overline{b}_{ij}\right)$.
\end{theorem}

\begin{remark}
 Special attention must be payed on the constants $P$, $Q_{ii}$, $R$ and $T_{ij}$, and the sign of
the square root, according to the selected model and the sample statistics $s_{ij}$ and
$\overline{b}_{ij}$.
\end{remark}

\subsubsection{Independent case}

For this case,
\begin{eqnarray*}
  \cov(b_{ij}) &=& \cov(\mathbf{e}_{i}^{T}\mathbf{B}\mathbf{e}_{j}) = \cov(\Vec\mathbf{e}_{i}^{T}\mathbf{B}\mathbf{e}_{j})
   = \cov((\mathbf{e}_{j} \otimes \mathbf{e}_{i})^{T}\Vec \mathbf{B}) \\
   &=& (\mathbf{e}_{j} \otimes \mathbf{e}_{i})^{T}\cov(\Vec \mathbf{B}) (\mathbf{e}_{j} \otimes \mathbf{e}_{i})\\
   &=& (\mathbf{e}_{j} \otimes \mathbf{e}_{i})^{T}\{(\mathbf{I}_{K^{2}}+\mathbf{K}_{K})[D \kappa_{0}
      (\mathbf{\Sigma}_{K}^{*} \otimes \mathbf{\Sigma}_{K}^{*})\\
   && + \ c_{0}(\boldsymbol{\mu}^{*}\boldsymbol{\mu}^{*T}  \otimes \mathbf{\Sigma}_{K}^{*} + \mathbf{\Sigma}_{K}^{*} \otimes
      \boldsymbol{\mu}^{*}\boldsymbol{\mu}^{*T})]\\
   && + \ D(\kappa_{0} - c^{2}_{0})\Vec \mathbf{\Sigma}_{K}^{*} \Vec^{T}\mathbf{\Sigma}_{K}^{*} \} (\mathbf{e}_{j} \otimes
      \mathbf{e}_{i}).
\end{eqnarray*}
Hence
\begin{equation}\label{vgwi}
    \cov(b_{ij}) = D\left[\kappa_{0}\sigma_{ii}\sigma_{jj} + (2\kappa_{0} - c_{0}^{2})\sigma_{ij}^{2}\right]
       + c_{0} \left(m_{jj}\sigma_{ii} + m_{ii}\sigma_{jj}+ 2m_{ij}\sigma_{ij}\right).
\end{equation}
From (\ref{egwd}),  by substituting $m_{ij} = s_{ij}-Dc_{0}
\sigma_{ij}$  in (\ref{vgwi}) we have
$$
  \cov(b_{ij}) = D (\kappa_{0}-2c_{0}^{2}) \sigma_{ii}\sigma_{jj} + D(2\kappa_{0} -3c_{0}^{2})\sigma_{ij}^{2}
       + c_{0} \left[\bar{b}_{jj}\sigma_{ii}+ \bar{b}_{ii} \sigma_{jj} + 2\bar{b}_{ij}\sigma_{ij}\right].
$$
Summarising

\begin{theorem}\label{th:7}
Assume that $\mathbf{y}_{d} \sim \mathcal{E}_{K}^{(K-1)}(\boldsymbol{\mu}_{d}, \mathbf{\Sigma};
h)$, independently, for $d = 1, 2, \dots,D$, such that
$$
  \mathbf{Y} = (\mathbf{y}_{1}|\mathbf{y}_{2}| \cdots |\mathbf{y}_{D} ) \mbox{ and }
  \boldsymbol{\mu} = (\boldsymbol{\mu}_{1}|\boldsymbol{\mu}_{2}| \cdots |\boldsymbol{\mu}_{D}
  ),
$$
and let
$$
    \mathbf{B} = \mathbf{Y}\mathbf{Y}^{T} = \sum_{d = 1}^{D} \mathbf{y}_{d}
    \mathbf{y}_{d}^{T}.
$$ Then, the method-of-moments estimators of $\mathbf{\Sigma}_{K}^{*}$ and
$\mathbf{M}=\boldsymbol{\mu}^{*}\boldsymbol{\mu}^{*T}$ are given by the following exact
expressions.

\textbf{For $i=1,2,\ldots,K$:}
\begin{equation}\label{sigmaiii}
    \widetilde{\sigma}_{ii}=\frac{\sqrt{Q_{ii}^{2}+4Ps_{ii}}-Q_{ii}}{2P},
\end{equation}
where $Q_{ii}^{2}+4Ps_{ii}\geq 0$, $P=D(3\kappa_{0}-5c_{0}^{2})$,
and $Q_{ii}=4c_{0}\overline{b}_{ii}$.
\begin{equation}\label{miiI}
    \widetilde{m}_{ii}=\overline{b}_{ii}-Dc_{0}\widetilde{\sigma}_{ii},
\end{equation}
where $\widetilde{\sigma}_{ii}$ has been previously found in (\ref{sigmaiii}).

If $P=0$,  then $\widetilde{\sigma}_{ii}=s_{ii}/Q_{ii}$.

\medskip

\textbf{For $i<j$, $i=1,\ldots,(K-1),j=2,\ldots,K$:}
\begin{equation}\label{sigmaiji}
    \widetilde{\sigma}_{ij}=\frac{\sqrt{c_{0}^{2}\overline{b}_{ij}^{2}-R(T_{ij}-s_{ij})}-c_{0}\overline{b}_{ij}}{R},
\end{equation}
where $\overline{b}_{ij}^{2}-R(T_{ij}-s_{ij})\geq 0$,
$R=D(2\kappa_{0}-3c_{0}^{2})$ and

$$
 T_{ij}=D(\kappa_{0}-2c_{0}^{2})\widetilde{\sigma}_{ii}\widetilde{\sigma}_{jj}+c_{0}\overline{b}_{jj}\widetilde{\sigma}_{ii}
+ c_{0}\overline{b}_{ii}\widetilde{\sigma}_{jj}.
$$
Here $\widetilde{\sigma}_{ii}$ and $\widetilde{\sigma}_{jj}$ were previously computed in
(\ref{sigmaiii}).
\begin{equation}\label{mii2I}
    \widetilde{m}_{ij}=\overline{b}_{ij}-Dc_{0}\widetilde{\sigma}_{ij},
\end{equation}
Denote the solution as
$$
(\widetilde{\mathbf{M}},\widetilde{\mathbf{\Sigma}}_{K}^{*}).
$$
Note that $s_{ij}=\widetilde{Cov(b_{ij})}$, this is $s_{ij}$ are obtained from the diagonal of
matrix $\mathbf{S}\in \mathfrak{R}^{K^{2}\times K^{2}}$.

If $R=0$, then $\widetilde{\sigma}_{ij}=\left(s_{ij}-T_{ij}\right)/\left(2c_{0}\overline{b}_{ij}\right)$.
\end{theorem}

\begin{remark}\label{rem3}
Recall that the method-of-moments estimators are not uniquely defined. In addition, if instead
of estimating the parameter $\theta$, method-of-moments estimator of, say, $g(\theta)$ is
desired, it can be obtained in several ways. One way would be to first find method-of-moments
estimator, say $\widetilde{\theta}$ of $\theta$ and then use $g(\widetilde{\theta})$ as an
estimator of $g(\theta)$. Alternatively, we can found the moments of function $g(\theta)$ and
then apply the method of moments to find the method-of-moments estimator
$\widetilde{g(\theta)}$ of $g(\theta)$. Estimators using either way are termed
method-of-moments estimators and may be not be the same in both cases, see \cite[Section 7.2.1,
p.276]{mgb:74}.
\end{remark}

The following result formalise the algorithm (Principal Coordinate Analysis, collected at
\cite{l:93}) for obtain $\boldsymbol{\mu}^{*}$, the estimated coordinates of the mean form (up
to translation, rotation, and reflection transformations) using the method-of-moments estimator
$\widetilde{\mathbf{M}}$.

\begin{theorem}\label{th:8}
Let $\widetilde{\mathbf{M}}$ the method-of-moments estimator of $\mathbf{M} =
\boldsymbol{\mu}^{*} \boldsymbol{\mu}^{*T}$ (for dependent or independent cases). Let
$\widetilde{\mathbf{M}} = \mathbf{V}_{1}\mathbf{L}\mathbf{V}_{1}^{T}$ is nonsingular part of
its spectral decomposition, where $\mathbf{V}_{1}$ is a semiorthogonal matrix, $\mathbf{V}_{1}
\in \Re^{K \times D}$ i.e. $\mathbf{V}_{1}^{T}\mathbf{V}_{1} = \mathbf{I}_{D}$ and $\mathbf{L}
= \diag(\lambda_{1}, \dots, \lambda_{D})$, with $D$ the rank of matrix
$\widetilde{\mathbf{M}}$. Then the method-of-moments estimator of $\boldsymbol{\mu}^{*}$ is
$$
  \widetilde{\boldsymbol{\mu}}^{*} = \mathbf{V}_{1}\mathbf{W},
$$
where $\mathbf{W} = \diag(\sqrt{\lambda_{1}}, \dots, \sqrt{\lambda_{D}})$.
\end{theorem}
\begin{proof}
It is follow from Remark \ref{rem3}.
\end{proof}

\begin{theorem}
Let $(\widetilde{\boldsymbol{\mu}}^{*}, \widetilde{\mathbf{\Sigma}}_{K}^{*})$ the method-of-moments
estimators of
$$
  (\boldsymbol{\mu}^{*}, \mathbf{\Sigma}_{K}^{*}).
$$
Then as $n \rightarrow \infty$
$$
  (\widetilde{\boldsymbol{\mu}}^{*}, \widetilde{\mathbf{\Sigma}}_{K}^{*})
  \rightarrow (\boldsymbol{\mu}^{*}, \mathbf{\Sigma}_{K}^{*})
  \qquad \mbox{in probability}.
$$
\end{theorem}
\begin{proof}
This follows from the consistency of the sample moments and the continuity of the function
$(\boldsymbol{\mu}^{*}, \mathbf{\Sigma}_{K}^{*})$ in $(E(\mathbf{B}), \cov(\Vec \mathbf{B}))$, see
\cite[Section 5d.1, p. 351]{r:73}.
\end{proof}

\section{Consistent estimation  when $\mathbf{\Sigma}_{D}$ is a general non-negative definite matrix}\label{sec4}

Results in this section are motivated in the result obtained by \cite{d:99} under a matrix
multivariate Gaussian distribution via the maximum likelihood estimation. We make an heuristic
evaluation of the useful of these results in our approach based in method-of-moments
estimation.

Our algorithm is based in the following modified expressions:
\begin{eqnarray}
      \widetilde{\mathbf{\Sigma}}_{D} &=& \frac{1}{nK} \sum_{i = 1}^{n}\left (\mathbf{X}^{c}_{i}-
      \widetilde{\boldsymbol{\mu}}^{*} \right )^{T} (\widetilde{\mathbf{\Sigma}}_{K}^{*})^{-}
      \left (\mathbf{X}^{c}_{i}- \widetilde{\boldsymbol{\mu}}^{*} \right ),\\
      \widetilde{\mathbf{\Sigma}}_{K}^{*} &=& \frac{1}{nD}\sum_{i = 1}^{n} \left (\mathbf{X}^{c}_{i}-
      \widetilde{\boldsymbol{\mu}}^{*} \right ) \widetilde{\mathbf{\Sigma}}_{D}^{-1}
      \left (\mathbf{X}^{c}_{i}- \widetilde{\boldsymbol{\mu}}^{*} \right )^{T}.
\end{eqnarray}

\noindent \textsc{Algorithm}

\noindent \textsc{Initialisation}:\\ $r = 0$; $\mathbf{\Sigma}_{K}^{* r} =
\widetilde{\mathbf{\Sigma}}_{K}^{*}$; $\mathbf{\Sigma}_{D}^{r} = \displaystyle\frac{1}{nK}
\sum_{i = 1}^{n}\left (\mathbf{X}^{c}_{i}- \widetilde{\boldsymbol{\mu}}^{*} \right )^{T}
(\mathbf{\Sigma}_{K}^{* r})^{-} \left (\mathbf{X}^{c}_{i}- \widetilde{\boldsymbol{\mu}}^{*}
\right)$;\\
$r = r + 1$
$$
  \mathbf{\Sigma}_{K}^{* r+1} = \frac{1}{nD}\sum_{i = 1}^{n} \left (\mathbf{X}^{c}_{i}-
  \widetilde{\boldsymbol{\mu}}^{*} \right )\left(\mathbf{\Sigma}_{D}^{r}\right)^{-1}
  \left (\mathbf{X}^{c}_{i}- \widetilde{\boldsymbol{\mu}}^{*} \right
      )^{T};
$$
$$
  \mathbf{\Sigma}_{D}^{r+1} = \frac{1}{nK} \sum_{i =1}^{n}\left (\mathbf{X}^{c}_{i}-
  \widetilde{\boldsymbol{\mu}}^{*} \right )^{T} (\mathbf{\Sigma}_{K}^{* r})^{-} \left (\mathbf{X}^{c}_{i}-
  \widetilde{\boldsymbol{\mu}}^{*} \right);
$$
\textsc{While}\\
$||\mathbf{\Sigma}_{D}^{r+1} - \mathbf{\Sigma}_{D}^{r}||_{2} > \varepsilon_{1}$ or
$||\mathbf{\Sigma}_{K}^{* r+1} - \mathbf{\Sigma}_{K}^{*r}||_{2} > \varepsilon_{2}$,\\
\textsc{Repeat}:\\
\begin{tabular}{l}
   $\qquad r = r+1$; \\
   $\qquad \mathbf{\Sigma}_{K}^{* r} = \mathbf{\Sigma}_{K}^{* r+1}$; \\
   $\qquad \mathbf{\Sigma}_{D}^{r} = \mathbf{\Sigma}_{D}^{r+1}$; \\
   $\qquad$\textsc{Recompute} $\mathbf{\Sigma}_{K}^{* r+1}$ and $\mathbf{\Sigma}_{D}^{r+1}$. \\
\end{tabular}\\
\textsc{Solutions are}:\\
$\widetilde{\mathbf{\Sigma}}_{K}^{*} = \mathbf{\Sigma}_{K}^{* r}$;
$\widetilde{\mathbf{\Sigma}}_{D}= \mathbf{\Sigma}_{D}^{r}$.

\medskip

Where $\varepsilon_{1}$ and $\varepsilon_{2}$ define two infinitesimal positive quantities and
$||\cdot||_{2}$ is the Euclidean norm, $\left( ||\mathbf{A}||_{2} =
\sqrt{\tr\left(\mathbf{AA}^{T}\right)}\right)$.

\begin{theorem} Let $(\widetilde{\boldsymbol{\mu}}^{*},
\widetilde{\mathbf{\Sigma}}_{K}^{*} \otimes \widetilde{\mathbf{\Sigma}}_{D})$ the
method-of-moments estimators of $(\mathbf{\Sigma}_{K}^{*},\boldsymbol{\mu}^{*} \otimes
\mathbf{\Sigma}_{D})$. Then as $n \rightarrow \infty$
$$
  (\widetilde{\boldsymbol{\mu}}^{*}, \widetilde{\mathbf{\Sigma}}_{K}^{*} \otimes \widetilde{\mathbf{\Sigma}}_{D})
  \rightarrow (\boldsymbol{\mu}^{*}, \mathbf{\Sigma}_{K}^{*} \otimes \mathbf{\Sigma}_{D})
  \qquad \mbox{in probability}.
$$
\end{theorem}
\begin{proof}
This follows from Remark \ref{rem3}.
\end{proof}

\section{Estimation of the form difference}\label{sec5}

A detailed discussion of Euclidean Distance Matrix, matrix form, form difference and their probabilistic,
geometrical, etc. properties may be found in \cite{l:91, l:93}. For your convenience, next we shall
introduce some notation, although in general we adhere to standard notation forms.

Consider the following square symmetric matrix, know as Euclidean Distance Matrix:
$$
  \mathbf{F}(\mathbf{X}) =
  \left (
  \begin{array}{ccccc}
    0 & d(1,2) & \dots & d(1,K-1) & d(1,K) \\
    d(2,1) & 0 & \dots & d(2,K-1) & d(2,K) \\
    \vdots & \vdots & \vdots & \ddots & \vdots \\
    d(K,1) & d(K,2) & \dots & d(K,K-1) & 0
  \end{array}
  \right ),
$$
where $d(i,j)$ denotes the Euclidean distance between landmarks $i$ and $j$, in shape theory
such matrix is termed \textit{form matrix}. Among others interesting properties of form matrix,
\cite{l:91} proves that $\mathbf{F}(\mathbf{X})$ is a maximal invariant under the group of
transformations consisting of translation, rotation, and reflection. Therefor,
$\mathbf{F}(\mathbf{X})$ retains all the relevant information about the form of an object.

Let $\mathbf{X}_{1}, \mathbf{X}_{2}, \dots, \mathbf{X}_{n}$ be $n$ independent observation from
population I and  $\mathbf{Y}_{1}, \mathbf{Y}_{2}, \dots, \mathbf{Y}_{m}$ be $m$ independent
observation from population II. Let the mean form of population I be
$\boldsymbol{\mu}^{\mathbf{X}}$ with the corresponding form matrix
$\mathbf{F}\left(\boldsymbol{\mu}^{\mathbf{X}}\right)$ and corresponding parameters for
population II be $\boldsymbol{\mu}^{\mathbf{Y}}$ and
$\mathbf{F}\left(\boldsymbol{\mu}^{\mathbf{Y}}\right)$. From \cite{l:93} we have the following
definition:
\begin{definition}
Form difference between population I and II is defined as
$$
  \mathbf{FDM}\left(\boldsymbol{\mu}^{\mathbf{X}}, \boldsymbol{\mu}^{\mathbf{Y}}\right) =
  \mathbf{F}\left(\boldsymbol{\mu}^{\mathbf{X}}\right) \ast
  \mathbf{F}\left(\boldsymbol{\mu}^{\mathbf{Y}}\right)^{-H},
$$
where $\ast$ denotes the Hadamard product, $0/0 = 0$ and $\mathbf{A}^{-H}$ denotes the inverse
of $\mathbf{A}$ with respect to the Hadamard product, a formula for such inverse in terms of
the usual product is given in \cite{clb:12}.
\end{definition}

From remark \ref{rem3}, the following theorem shows that the form difference between two
populations can be estimated consistently when landmarks are perturbed dependently along each
axis but independently or not correlated between the axes.
\begin{theorem}
Let $\left(\boldsymbol{\mu}^{\mathbf{X}}, \mathbf{\Sigma}_{K\mathbf{X}}^{*} \otimes
\mathbf{\Sigma}_{D\mathbf{X}} \right)$ and $\left(\boldsymbol{\mu}^{\mathbf{Y}},
\mathbf{\Sigma}_{K\mathbf{Y}}^{*} \otimes \mathbf{\Sigma}_{D\mathbf{Y}} \right)$ be the
parameters for the two populations. If $\mathbf{\Sigma}_{D\mathbf{X}} =
\mathbf{\Sigma}_{D\mathbf{Y}} = \mathbf{I}_{D}$, then
$$
  \widetilde{\mathbf{FDM}}\left(\widetilde{\boldsymbol{\mu}}^{\mathbf{X}}, \widetilde{\boldsymbol{\mu}}^{\mathbf{Y}}\right) =
  \widetilde{\mathbf{F}}\left(\widetilde{\boldsymbol{\mu}}^{\mathbf{X}}\right) \ast
  \widetilde{\mathbf{F}}\left(\widetilde{\boldsymbol{\mu}}^{\mathbf{Y}}\right)^{-H}
  \Rightarrow \mathbf{FDM}\left(\boldsymbol{\mu}^{\mathbf{X}},
  \boldsymbol{\mu}^{\mathbf{Y}}\right) \quad \mbox{in probability.}
$$
\end{theorem}

\begin{theorem} Let $\left(\boldsymbol{\mu}^{\mathbf{X}},
\mathbf{\Sigma}_{K\mathbf{X}}^{*} \otimes \mathbf{\Sigma}_{D\mathbf{X}} \right)$ and
$\left(\boldsymbol{\mu}^{\mathbf{Y}}, \mathbf{\Sigma}_{K\mathbf{Y}}^{*} \otimes
\mathbf{\Sigma}_{D\mathbf{Y}} \right)$ be the parameters for the two populations. Then
$$
  \widetilde{\mathbf{FDM}}\left(\widetilde{\boldsymbol{\mu}}^{\mathbf{X}}, \widetilde{\boldsymbol{\mu}}^{\mathbf{Y}}\right) =
  \widetilde{\mathbf{F}}\left(\widetilde{\boldsymbol{\mu}}^{\mathbf{X}}\right) \ast
  \widetilde{\mathbf{F}}\left(\widetilde{\boldsymbol{\mu}}^{\mathbf{Y}}\right)^{-H}
  \rightarrow \mathbf{FDM}\left(\boldsymbol{\mu}^{\mathbf{X}},
  \boldsymbol{\mu}^{\mathbf{Y}}\right) \quad \mbox{in probability.}
$$
\end{theorem}

\section{Example}\label{exa}

The mouse vertebra problem was originally studied in the Gaussian case by \cite{DM98} (see also
\cite{md:89}). A further analysis under elliptical models was implemented by \cite{dc:12}. The experiment
considers the second thoracic vertebra T2 of two groups of mice: large and small. The mice are selected
and classified according to large or small body weight; in this case, the sample consists of 23, 23 and
30 large, small and control bones, respectively. The vertebras are digitised and summarised in six
mathematical landmarks which are placed at points of high curvature, see figure
\ref{fig:Smalllargeandcontrolmousevertebrasample}; they are symmetrically selected by measuring the
extreme positive and negative curvature of the bone. See \cite{DM98} for more details. The shape
difference analysis among the three groups is quite solved by a different approaches. However the
correlation structure among landmarks requires more analysis; strong assumptions about those relations
are usually considered because the complex exact shape distribution and a non existence theory for
estimation for such invariant functions.

More than an example, this landmark data  is highly valuable for a correlation structure analysis because
the symmetry of the vertebra, certainly suggest a priori a non isotropic model. The control group is also
useful for comparisons and correctness.
\begin{figure}[!h]
  \begin{center}
   \includegraphics[width=8cm,height=8cm]{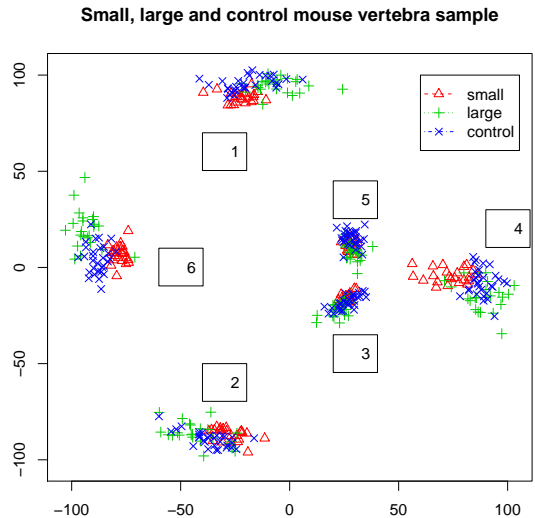}%
   \caption{Mouse vertebra sample}\label{fig:Smalllargeandcontrolmousevertebrasample}
  \end{center}
\end{figure}

Theorems \ref{th:6} and \ref{th:8} can be easily  implemented for a number of models. We focus on the
main novelty (Theorem \ref{th:6}) and Kotz type model (including Gaussian) which is very flexible and
meaningful for various values of the parameters $r, s$ and $N$, see appendix.

First of all we illustrate Theorem \ref{th:7} under six different models with independent landmarks.
Moment-method estimates of mean shape by using the common Gaussian model is shown in figure
\ref{GaussianSLC}, in this case the estimate are complete unrealistic, as we expect, given the assumption
of independence of landmarks. However, if we consider more complex models based on independence, the
estimation tends to be more similar to the structure suggested by the sample. The addressed evolution
from Kotz 1 to Kotz 5 is depicted in figures \ref{Kotz1SLC} to \ref{Kotz5SLC}.

\begin{figure}[!h]
  \begin{center}
   \includegraphics[width=6cm,height=6cm]{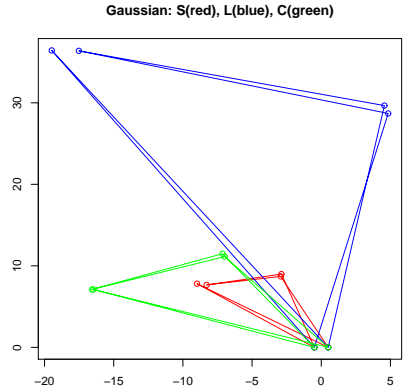}%
   \caption{Moment method estimates under independence: Gaussian model}\label{GaussianSLC}
  \end{center}
\end{figure}

\begin{figure}[!h]
  \begin{center}
   \includegraphics[width=6cm,height=6cm]{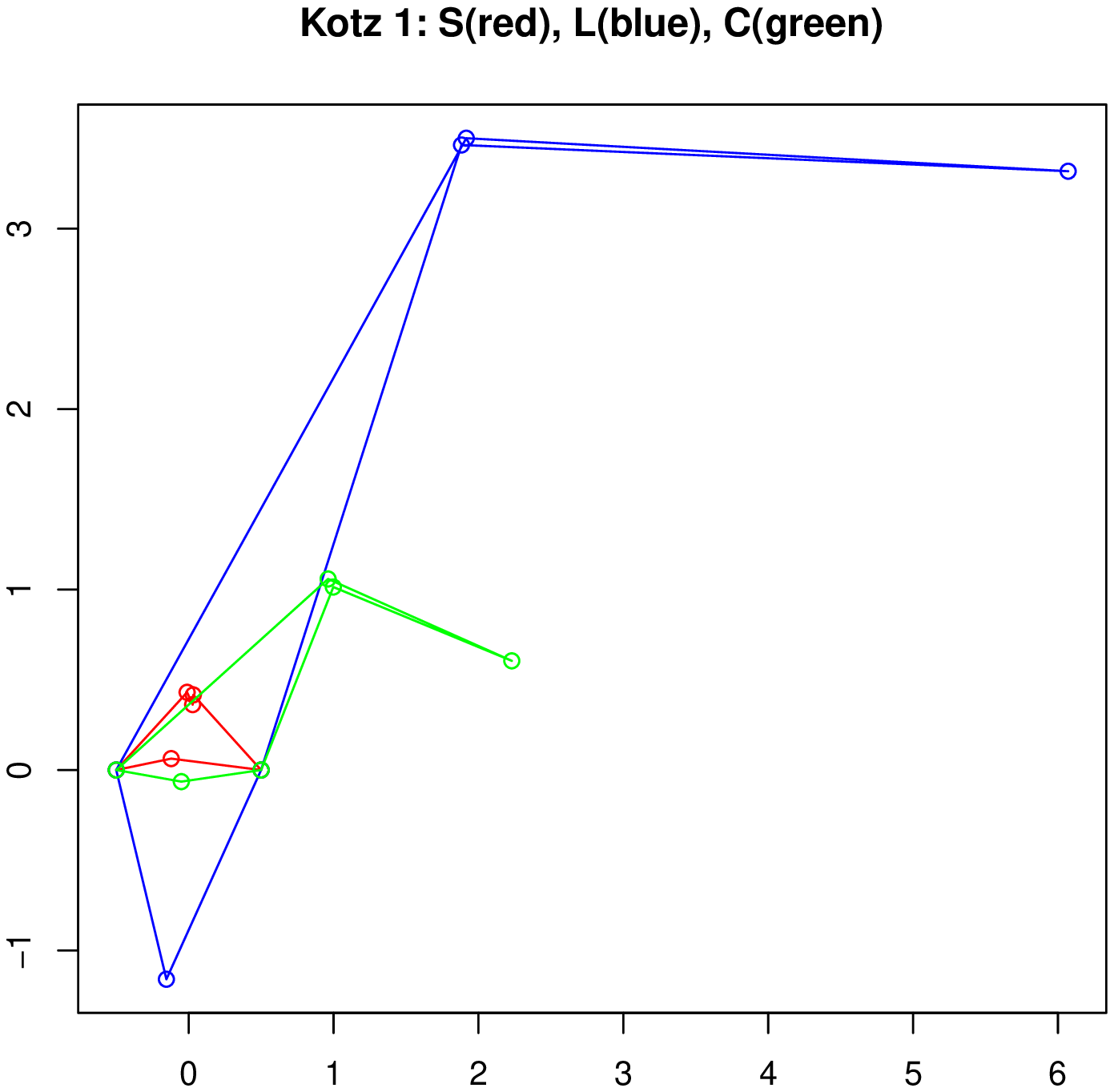}%
   \caption{Moment method estimates under independence: Kotz 1 model}\label{Kotz1SLC}
  \end{center}
\end{figure}

\begin{figure}[!h]
  \begin{center}
   \includegraphics[width=6cm,height=6cm]{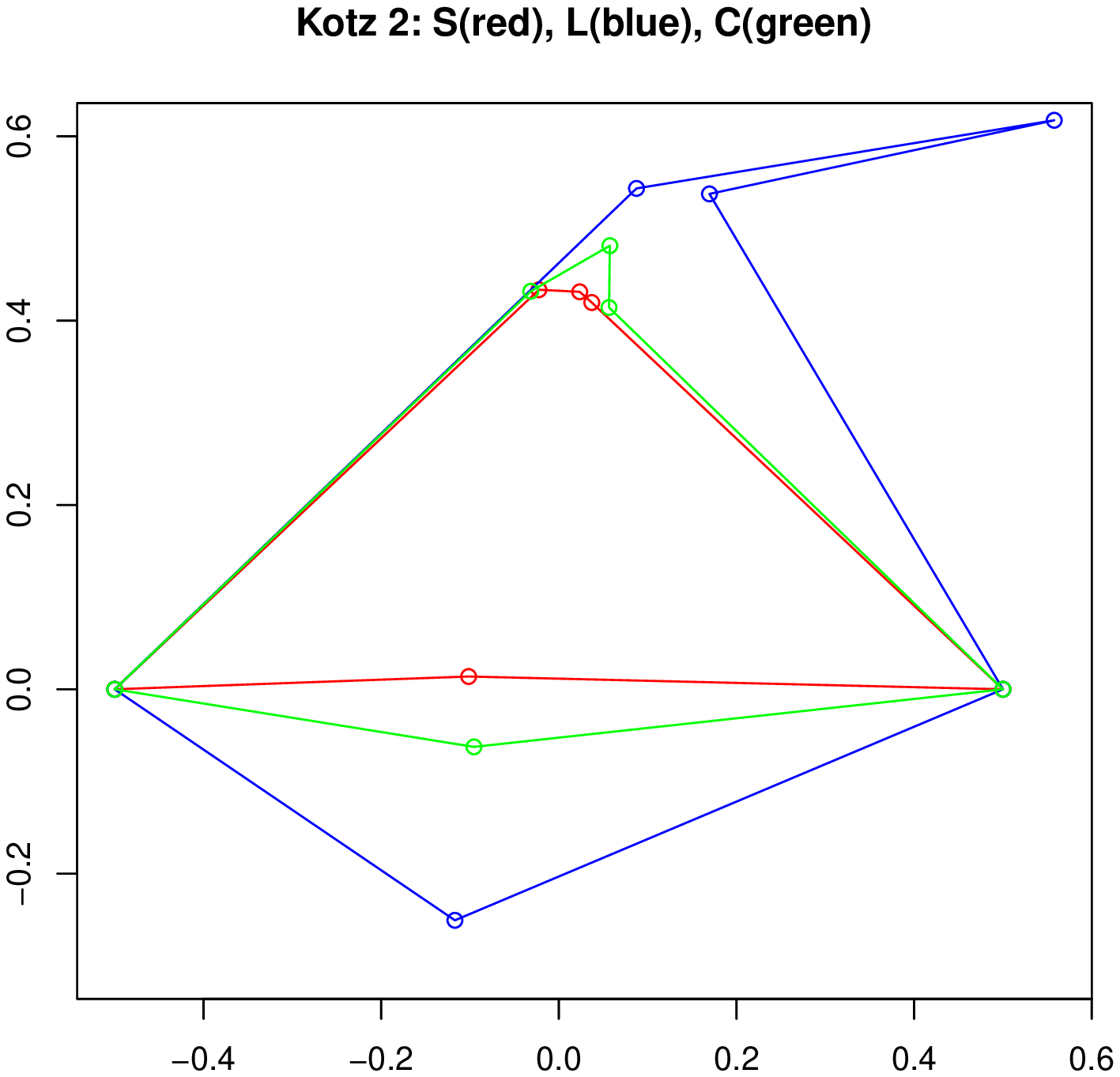}%
   \caption{Moment method estimates under independence: Kotz 2 model}\label{Kotz2SLC}
  \end{center}
\end{figure}

\begin{figure}[!h]
  \begin{center}
   \includegraphics[width=6cm,height=6cm]{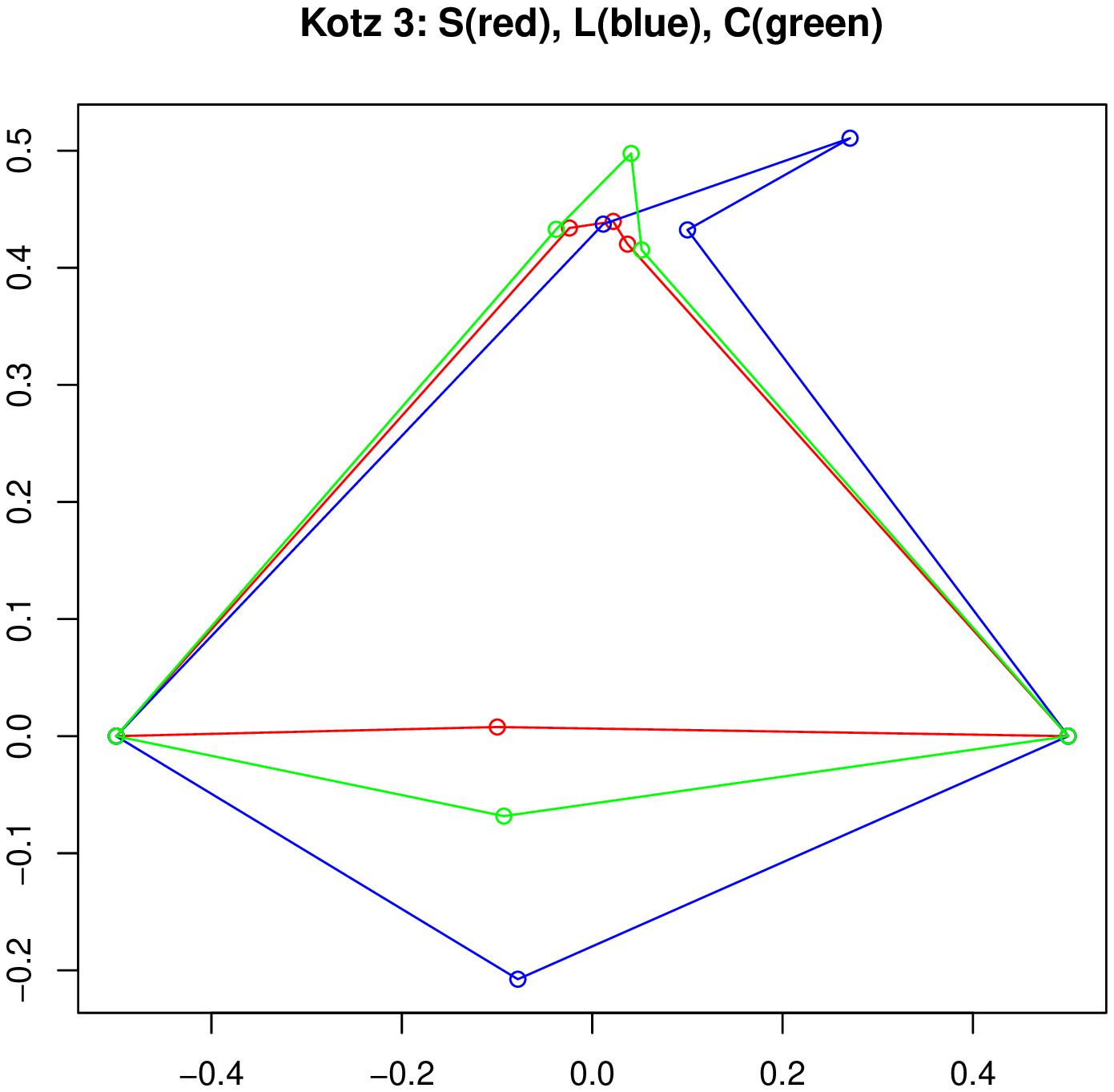}%
   \caption{Moment method estimates under independence: Kotz 3 model}\label{Kotz3SLC}
  \end{center}
\end{figure}

\begin{figure}[!h]
  \begin{center}
   \includegraphics[width=6cm,height=6cm]{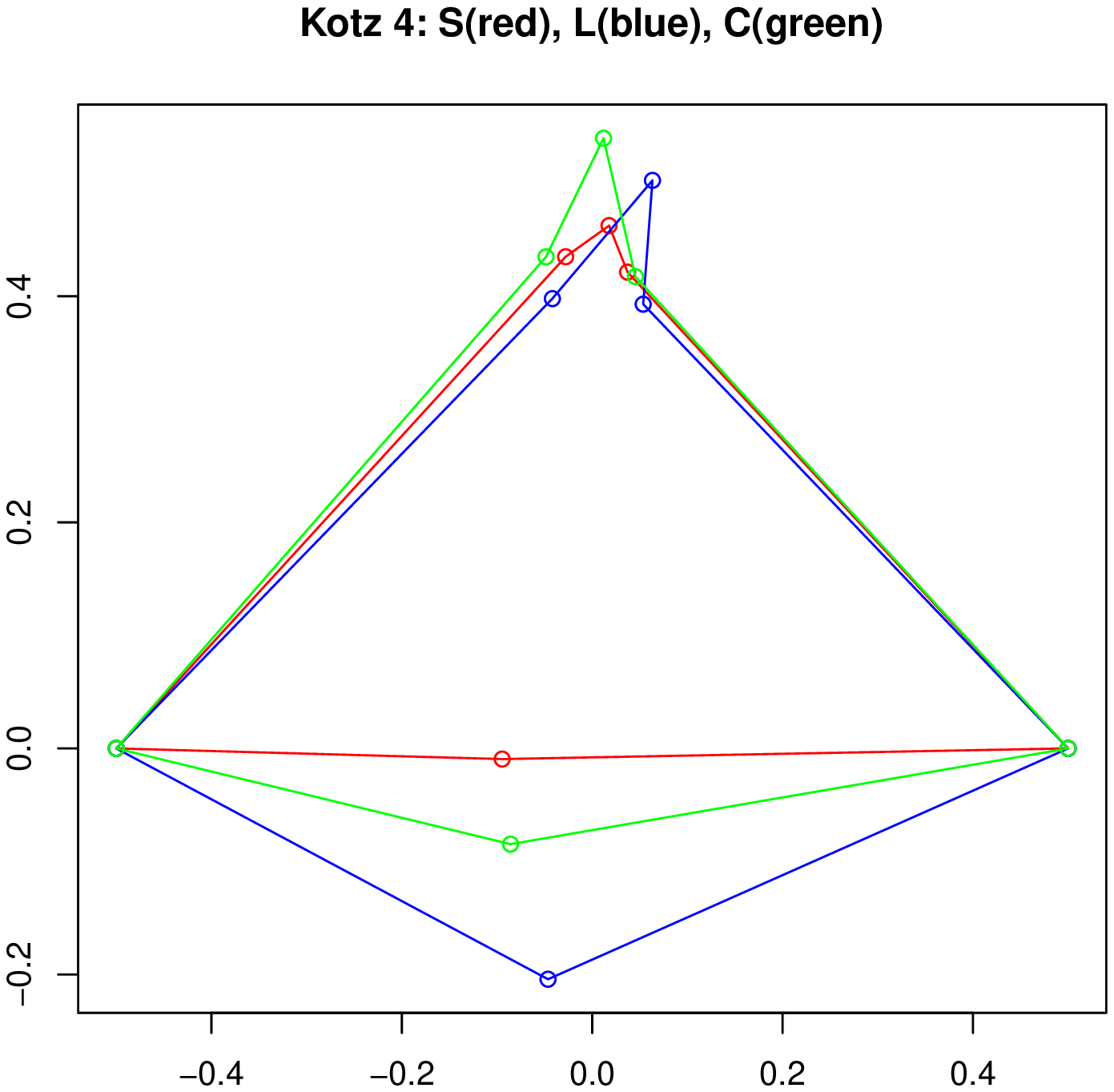}%
   \caption{Moment method estimates under independence: Kotz 4 model}\label{Kotz4SLC}
  \end{center}
\end{figure}

\begin{figure}[!h]
  \begin{center}
   \includegraphics[width=6cm,height=6cm]{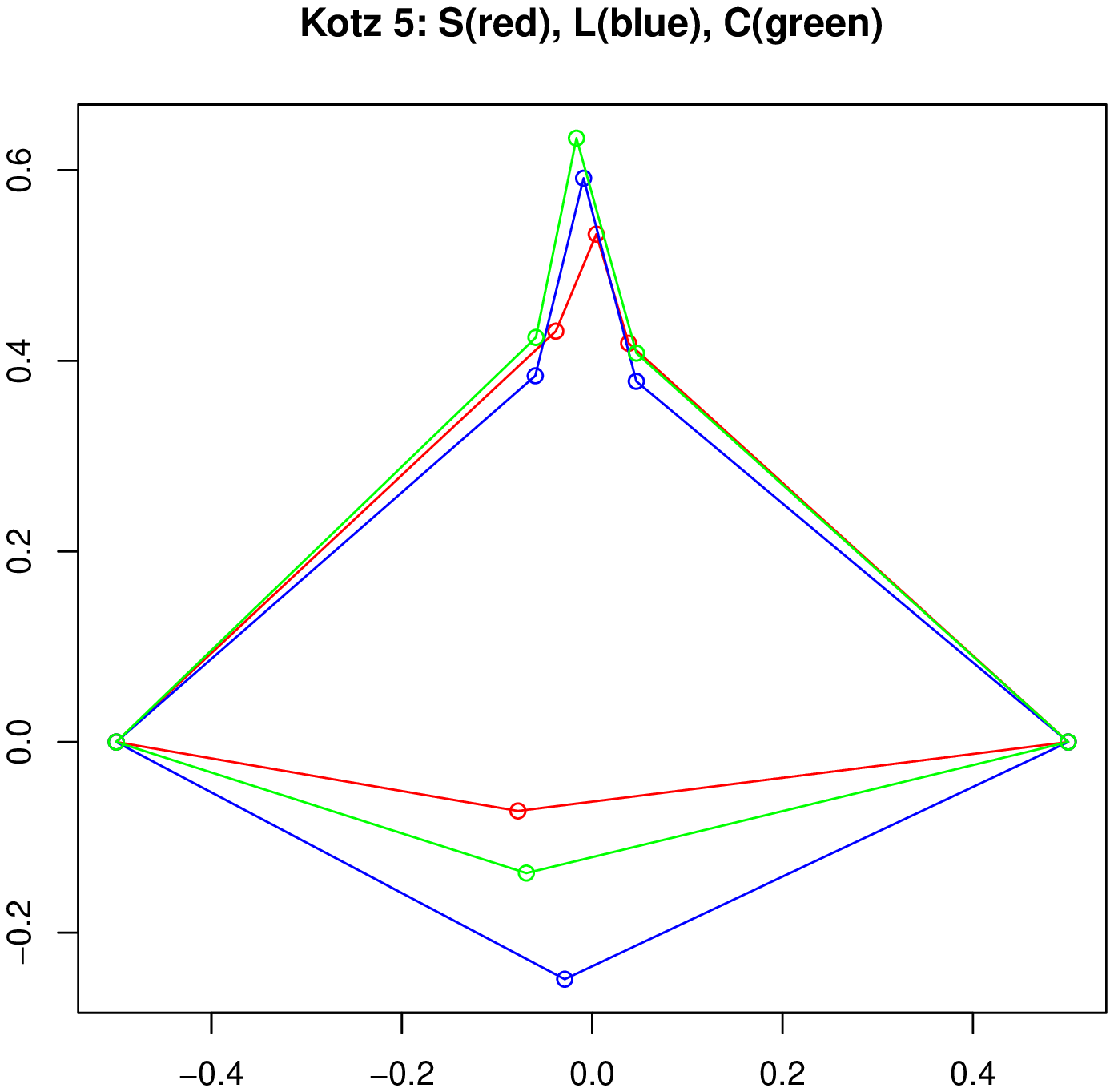}%
   \caption{Moment method estimates under independence: Kotz 5 model}\label{Kotz5SLC}
  \end{center}
\end{figure}

An heuristic behavior is noted, the lack of dependence in the Gaussian model, and its unrealistic moment
method estimates, it seems to be improved by considering a more robust Kotz type model even with landmark
dependence. The literature has studied this artificial data by the independent Gaussian case, so, given
that no expert have set this assumption we can get further into more robust analysis and advance in some
selection criteria, but if we have an experiment modeled by the independent Gaussian case according to
the opinion of an expert in the field, we must follow that law and the further analysis, based on
landmark dependence and elliptical families, that we provide next, cannot be implemented in such cases.

In the artificial mice data, we now can  focus on the dependent case and the moment method estimators of
Theorem \ref{th:6}, given that the Gaussian case is out of any consideration, then we have to study, for
example, other Kotz models. In order to illustrate the important effect of landmark dependence we
consider the simplest Kotz model after Gaussian, when $N=2$, $r=1/2$ and $s=1$, which is referred as Kotz
1 model, and we compare the performance of Theorem \ref{th:6} with another mean shape estimations. Table
\ref{table2} provides comparisons among mean shape estimates of the small group, they include mean shape
by moments of Theorem \ref{th:6}, the mean shape by  Frechet method (see \cite{k:92}), and Bookstein
method (see \cite{b:86}); certainly the estimations are truly similar. Note also that Kotz 1 law with
independent landmarks provided a bad moment method estimator of mean shape, but the same model under the
expected and realistic dependence revels  similarity with more complex mean shape estimators derived by
standard shape theories, see figures \ref{Kotz1SLC} and \ref{Kotz1th6}, respectively.

\begin{table}[!h]  \centering \caption{Estimation for the   mean shape for the small group by Theorem
\ref{th:6} (Kotz 1), Frechet (F), and Bookstein (B).}\label{table2}
\medskip
\begin{tiny}
\begin{tabular}{|c||c||c||c||c||c|}
  \hline
Th. 6, $\widetilde{\mu}_{1}$&Th. 6, $\widetilde{\mu}_{2}$&B. $\widetilde{\mu}_{1}$&B.
$\widetilde{\mu}_{2}$&F. $\widetilde{\mu}_{1}$&F. $\widetilde{\mu}_{2}$\\\hline
 -0.5 & 0&-0.5 & 0&-0.5 & 0\\
 0.5 & 0&0.5 & 0&0.5 & 0\\
  0.084507028 & 0.3301634&0.08469746&  0.2933430& 0.08490820 & 0.2924684\\
  0.014836162 & 0.6957339&0.01215768 & 0.5613175&0.01245608 & 0.5589496\\
 -0.073397569  &0.3394693&-0.06874750 & 0.2991278&-0.06869796 & 0.2982314\\
 -0.005026754 &-0.2184060&-0.02502185& -0.3041418&-0.02512807& -0.3044915\\
  \hline
\end{tabular}
\end{tiny}
\end{table}

The exact formula for the moments estimation (Theorem \ref{th:6}) also agrees with the previous
conclusions in literature about strong difference in Gaussian mean shape  between the small (S) and large
(L) groups. Figure \ref{Kotz1th6} also shows the mean shape estimation of the control (C) group. As we
expect, the control group must tend to show strong symmetry among landmarks, by "averaging" in some sense
the small and large estimates.

\begin{figure}[!h]
  \begin{center}
   \includegraphics[width=6cm,height=6cm]{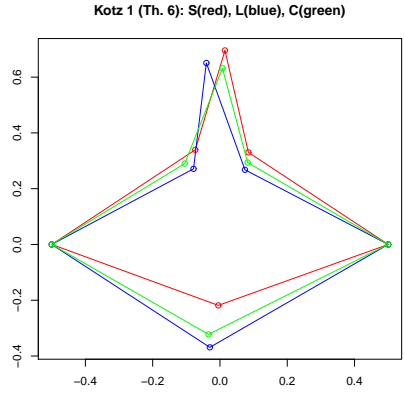}%
   \caption{Moment method estimates under dependence: Kotz 1 model}\label{Kotz1th6}
  \end{center}
\end{figure}

Different types of Kotz distribution have also modeled the sample, they correspond to the denoted models
Kotz 1, Kotz 2, Kotz 3, Kotz 4 and Kotz 5, with parameters $N=2, s=1, r=1/2$; $N=3, s=1, r=1/2$; $N=2,
s=2, r=1/2$; $N=2, s=3, r=1/2$ and $N=20, s=20, r=1/2$, respectively. Technical details about the
generalised singular Pseudo-Wishart distributions and particular Kotz Pseudo-Wishart distributions
referred in this example, can be seen in the appendix. The corresponding mean shapes estimates were
computed, but for reasons of space, we only show the results of the Kotz 5 model (suggested by the
preceding independent results and certain selection criteria that we will propose later) see figure
\ref{Kotz5}.
\begin{figure}[!h]
  \begin{center}
   \includegraphics[width=6cm,height=6cm]{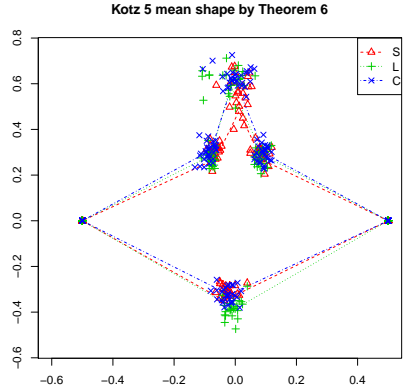}%
   \caption{Mean shape estimates for large, small and control groups under the Kotz 5 model}\label{Kotz5}
  \end{center}
\end{figure}

Now we apply the algorithm for a consistent estimation when $\mathbf{\Sigma}_{D}$ is a general
non-negative definite matrix, under  the Kotz 5 law. For the routine propose in Section \ref{sec4} we
have fixed $\varepsilon_{1}=\varepsilon_{2}=0.000005$ in the three groups small, large  and control,
then we found that the number of iteration to reach the addressed tolerance is 57, 53 and 61,
respectively.

For the small group the estimated covariance matrices are given next (here the associated correlation
matrix $\boldsymbol{\rho}$ of $\mathbf{\Sigma}$ is provided, for the sake of interpretation, recall that
$\boldsymbol{\rho}= \left(\text{diag}(\mathbf{\Sigma})\right)^{-\frac{1}{2}} \, \mathbf{\Sigma} \,
\left(\text{diag}(\mathbf{\Sigma})\right)^{-\frac{1}{2}}$):

\begin{scriptsize}
$$
  \widetilde{\boldsymbol{\rho}}_{K}^{*}=
  \left(
  \begin{array}{cccccc}
 1.0000000& -0.88123087& -0.45286210& -0.0947470&  0.3167573&  0.1049686\\
-0.8812309&  1.00000000& -0.01931031& -0.3859582& -0.7246992&  0.3741399\\
-0.4528621& -0.01931031&  1.00000000&  0.9267571&  0.6990041& -0.9323209\\
-0.0947470& -0.38595825&  0.92675709&  1.0000000&  0.9113738& -0.9979133\\
 0.3167573& -0.72469917&  0.69900414&  0.9113738&  1.0000000& -0.9087033\\
 0.1049686&  0.37413987& -0.93232089& -0.9979133& -0.9087033&  1.0000000\\
  \end{array}
  \right);
$$
\end{scriptsize}

and
\begin{scriptsize}
$$
  \widetilde{\boldsymbol{\rho}}_{D}=
  \left(
  \begin{array}{cc}
    1.00000000& -0.1305434\\
    -0.1305434&  1.00000000\\
  \end{array}
  \right).
$$
\end{scriptsize}

For the large group the estimated correlation matrices are:

\begin{scriptsize}
$$
  \widetilde{\boldsymbol{\rho}}_{K}^{*}=
  \left(
  \begin{array}{cccccc}
  1.00000000& -0.65790585& -0.49909037& -0.05610956& -0.07744806&  0.42376772\\
 -0.65790585&  1.00000000&  0.06499537& -0.44441269& -0.29572585&  0.03327717\\
 -0.49909037&  0.06499537&  1.00000000&  0.42647926&  0.65775361& -0.76574318\\
 -0.05610956& -0.44441269&  0.42647926&  1.00000000&  0.65493138& -0.73668655\\
 -0.07744806& -0.29572585&  0.65775361&  0.65493138&  1.00000000& -0.79064351\\
  0.42376772&  0.03327717& -0.76574318& -0.73668655& -0.79064351&  1.00000000\\
  \end{array}
  \right);
$$
\end{scriptsize}

and
\begin{scriptsize}
$$
  \widetilde{\boldsymbol{\rho}}_{D}=
  \left(
  \begin{array}{cc}
    1.00000000& -0.2080039\\
   -0.2080039&  1.00000000\\
  \end{array}
  \right).
$$
\end{scriptsize}

Meanwhile in the control group the estimated correlation matrices
are:

\begin{scriptsize}
$$
  \widetilde{\boldsymbol{\rho}}_{K}^{*}=
  \left(
  \begin{array}{cccccc}
  1.00000000& -0.6179802& -0.6137305& -0.3968700&  0.05198526&  0.5036286\\
 -0.61798019&  1.0000000& -0.1036402& -0.4037052& -0.70811391&  0.2660900\\
 -0.61373047& -0.1036402&  1.0000000&  0.7604036&  0.50488092& -0.8386367\\
 -0.39687003& -0.4037052&  0.7604036&  1.0000000&  0.64810504& -0.9587295\\
  0.05198526& -0.7081139&  0.5048809&  0.6481050&  1.00000000& -0.6427126\\
  0.50362856&  0.2660900& -0.8386367& -0.9587295& -0.64271257&  1.0000000\\
  \end{array}
  \right);
$$
\end{scriptsize}
and
\begin{scriptsize}
$$
  \widetilde{\boldsymbol{\rho}}_{D}=
  \left(
  \begin{array}{cc}
    1.00000000& 0.1048453\\
   0.1048453&  1.00000000\\
  \end{array}
  \right).
$$
\end{scriptsize}

The three groups reveal almost null correlation among axes, but some important correlation among
landmarks, as we expect from the pseudo-symmetry of the bones. The estimates in the small and large
groups detects the main landmarks responsible for the mean shape difference, meanwhile in the control
case the estimates tends to follow the main contribution of
large or small differentiating landmarks as we expect.

\bigskip

In a similar way we have run the routines with the same tolerance
$\varepsilon_{1}=\varepsilon_{2}=0.000005$ for the models Kotz 1, to Kotz 4; they reached the stability
between 50 to 70 iterations in the three groups, and similar conclusions about the almost null
correlation among axes and strong correlation among landmarks were found in the models. We will not show
the estimates of each Kotz type, but have provided the results for the model Kotz 5 type, for reasons
that we will explain later when the "best" model is selected under certain criteria; it was also
suggested by the independent case analysis.

For a selection model criteria, the control group plays a fundamental role, in this case we just need to
look for the law which obtains the minimum coefficient of variation when the small and large groups are
compared with the control one; the analysis also must consider the distance between small and the large
group relative to the mean with controls. We apply non-Euclidian distance between covariance matrix, a
technique due to \cite{DKZ09}. The method is appropriate for meaningful correlation matrices, in this
case it is performed only for $\widetilde{\Sigma}_{K}^{*}$, because $\widetilde{\Sigma}_{D}$ certainly
ratifies in all the models that no correlation among axes is observed. In Tables \ref{table3} and
\ref{table4},  $K1$,..., $K5$, $s$, $l$, $c$, stand for Kotz 1,..., Kotz 5, small, large and control,
respectively.

\bigskip

\begin{table}[!h]  \centering \caption{Selection model criteria.}\label{table3}\medskip
\begin{footnotesize}
\begin{tabular}{|c|c|c|c|c|c|c|c|c|}
  \hline
    & K1l & K1c & K2s & K2l & K2c & K3s & K3l & K3c \\\hline
  K1s & 12.9&  8.7(37)& 12.8& 15.9& 14.1& 11.8& 15.7& 13.7\\
  K1l & &  5.1(37)& 10.6 & 6.1 & 8.4 &10.8 & 5.1 & 8.5 \\
  K1c & &  &  9.6 & 9.1&  8.9 & 9.2 & 8.6 & 8.7 \\\hline
  K2s  &&  &  & 11.0 &11.0(14) & 6.0& 11.1 &10.7 \\
  K2l & &  &  &  &  9.0(14)& 12.0 & 1.5 & 9.1 \\
  K2c & &  &  &  &  & 11.6 & 8.8 & 0.9 \\\hline
  K3s & &  &  &  &  &  & 12.1& 11.2(15) \\
  K3l & &  &  &  &  &  &  &  9.0(15) \\\hline
\end{tabular}
\end{footnotesize}
\end{table}
\bigskip

\begin{table}[!h]  \centering \caption{Selection model criteria.}\label{table4}
\medskip
\begin{footnotesize}
\begin{tabular}{|c|c|c|c|c|c|c|}
  \hline
    &  K4s & K4l & K4c & K5s & K5l & K5c \\\hline
  K1s &   11.1&  15.0&  13.1&  11.1&  13.9&  12.1\\
  K1l &  11.3 &  3.9 &  8.6 & 12.2 &  2.4 &  2.8\\
  K1c &  9.2&   7.7 &  8.4 &  9.7 &  6.4 &  4.8\\\hline
  K2s  &  8.9&  10.8 & 10.3 & 10.9 & 10.1 &  9.7\\
  K2l & 13.3 &  3.0 &  9.5&  14.7&   4.4 &  7.0\\
  K2c & 12.4 &  8.5 &  2.1 & 13.4 &  8.2 &  8.4\\\hline
  K3s & 6.4 & 11.6 & 10.7 &  9.2&  10.7 &  9.8\\
  K3l & 13.2 &  1.6 &  9.5 & 14.6 &  3.3 &  6.2\\
  K3c &  11.9 &  8.7 &  1.2 & 12.9 &  8.4 &  8.3\\\hline
  K4s &    &  12.6 & 11.4(16) &  6.8  &11.6 & 10.6\\
  K4l &    &   &   9.1(16) & 13.9  & 1.8 &  5.1\\
  K4c &    &   &   &  12.3  & 8.7 &  8.3\\\hline
  K5s  &   &   &   &   &  12.8 & 11.6(74)\\
  K5l  &   &   &   &   &   &   3.6(74)\\
  \hline
\end{tabular}
\end{footnotesize}
\end{table}
\bigskip

Tables \ref{table3} and \ref{table4} shows all the pairwise covariance distances, in particular, the
percentage variation coefficient is presented in parenthesis. We are searching for models which reflect
the role of the control group and separated the classes properly, the analysis must be complemented with
mean shape estimates and a third criteria involving how the last ones are far from another accepted
estimates, the Frechet mean shape for example. The addressed mean shape distance can be achieved by a
number of approaches, see for example \cite{ken:84}.

Kotz 3 and 4 laws behave well with percentage variation coefficient, but the corresponding distance  with
the control group and the sample is to far to be realistic, specially with the small group.  If we find
the so called Riemannian distance among the moment method estimates and Frechet and Bookstein mean shape,
we obtain the results of table \ref{table4b}:

\begin{table}[!h]  \centering \caption{Model selection criteria.}\label{table4b}\medskip
\begin{small}
\begin{tabular}{|c|c|c|c|c|c|c|}
  \hline
    &  K2& K3& K4& K5& F.& B.\\\hline
  K1&  0.274& 0.236& 0.211& 0.180& 0.113& 0.112\\
  K2&  & 0.082& 0.128& 0.153& 0.189& 0.191\\
  K3&  & & 0.048& 0.078& 0.131& 0.133\\
  K4&  & & & 0.035& 0.099& 0.100\\
  K5&  & & & & 0.067& 0.068\\
  F&  & & & & & 0.002\\
  \hline
\end{tabular}
\end{small}
\end{table}

\bigskip
The mean shape estimate based on a Kotz 5 model is very near to the estimates computed by Frechet and
Bookstein (which are significantly similar), it also reflects good difference between the small and
large; similar findings for the large and small group were computed, then collecting the results, we can
propose Kotz type 5 model as a suitable law for modeling this particular example. Note that the selection
agrees with the conclusion proposed in the independent case.

It is important to note, that mathematical or statistical selection is just a suggestion for an
experiment lacking of any prior assumption of the supporting distribution provided by an expert. In our
case, literature shows no expert assumption about normality, in fact, this data full studied in
\cite{DM98} and the references therein, was traditionally set in the Gaussian theory in order to simplify
computations and/or the use of the classes of exact distributions were not available at that time.
However, if an experiment was sufficiently studied by an expert which the Gaussian model is truly normal,
then the above selection of models, are out of significance; and given that moments-method estimates does
not work under Gaussianity, as we have shown in this example, then the results presented here cannot be
applied properly.

Now, at this stage, the conclusion about Kotz 5 model ratifies that non-Gaussian models explain better
the three samples (an elliptical isotropic approach also verified this conclusion, see for example
\cite{dc:12}).

Once the model is selected, we are interested in application of Section \ref{sec5}, about estimation of
mean form difference. In fact, we can go further by considering hypothesis testing for equality of the
associated Euclidean Distance Matrices of two populations.

The methodology can be found in \cite{lr:91} and the references therein. We are interested in testing
$H_{0}: \mathbf{F}(\boldsymbol{\mu}^{\mathbf{X}})=c\mathbf{F}(\boldsymbol{\mu}^{\mathbf{Y}})$, for some
$c>0$, where $\boldsymbol{\mu}^{\mathbf{X}}$ and $\boldsymbol{\mu}^{\mathbf{Y}}$ are the population mean
shape. Based on a sample of objects $\mathbf{X}$'s and $\mathbf{Y}$'s, with corresponding estimated mean
shapes $\widetilde{\boldsymbol{\mu}}^{X}$ and $\widetilde{\boldsymbol{\mu}}^{Y}$ obtained with the exact
formula given in Theorem \ref{th:6}, we derive the form difference matrix
$\mathbf{FDM}\left(\widetilde{\boldsymbol{\mu}}^{\mathbf{X}},
\widetilde{\boldsymbol{\mu}}^{\mathbf{Y}}\right)$. This last matrix can be used for defining a number of
suitable statistics for testing $H_{0}$, however, \cite{lr:91} recommend the following:
$$
  T=\max_{i,j}FDM_{ij}/\min_{i,j}FDM_{ij},
$$
where $FDM_{ij}$ is the $i,j-$element of matrix $\mathbf{FDM}$. Note that if $H_{0}$ is true $T$ is close
to one. Moreover, $T$ satisfies the desirable property of invariance under scaling, see \cite{lr:91} for
more details.

The null distribution is difficult to obtain even in the simplest case of Gaussian, so we can obtain an
empirical null distribution by using the well known bootstrap procedure, see \cite{lr:91} and the
references therein. For similar samples of the current example, the above referred authors recommend a
bootstrap of size 100.

Once the empirical distribution is obtain, a p-value, based on the upper tail of the observed statistics,
rejects $H_{0}$ for small values near to 0.1.

Table \ref{table5} reports such tests for the Gaussian and Kotz 3 type models and the three pair
comparisons of interest.  We note that the usual Gaussian case, under the expected dependence condition
of Theorem \ref{th:6} malfunction and cannot detect the role of the control test, given a wrong
conclusion. The selected model by covariance distances, separates as we expect the control group and
gives a suitable p-value of certain difference, but it is not sufficient enough for concluding shape
difference. This open an interesting discussion about the method based on coordinate free approach of
\cite{lr:91}, given that the quotient pairwise-element in definition of the matrix form difference is
neglecting the whole matrix structure. Improving this aspect deserves a further work by defining a more
robust matrix $\mathbf{FDM}$ based on usual products than the very restrictive Hadamard product.
Moreover, finding the corresponding exact distribution of $\mathbf{FDM}(\mathbf{X},\mathbf{Y})$ can
provide a promising null distribution which can model hypothesis testing efficiently.

\begin{table}
\centering
\caption{p-values of testing the equality of mean shape under different models and pairs of populations .}\label{table5}
\begin{tabular}{|c||c|c|c|}
  \hline
   & Small-Large & Small-Control & Large-Control  \\\hline
  Gaussian & 0.00 & 0.00 & 0.00 \\
  Kotz 5 & 0.12 & 0.51 & 0.74 \\
  \hline
\end{tabular}
\end{table}

\section{Conclusions}
\begin{enumerate}
  \item First, by replacing the Gaussian model for the elliptical model, an infinite range of possibilities
    in making an assumption of a model is opened, allowing to model a wide range of real situations,
    more or less heavy tails and more or less kurtosis than the Gaussian model.
  \item Under this family of elliptical models is possible consistently estimate all parameters.
   \item Notably, all these estimators are extremely easy to calculate.
  \item Alternatively to the hypothesis assumed in subsections 2.3 and 2.4, an interesting alternative to investigate is:
    Assume that the joint distribution of $\mathbf{E}_{1}, \mathbf{E}_{2}, \dots,\mathbf{E}_{n}$ is
    $$
      \mathbb{E} = (\mathbf{E}_{1}, \mathbf{E}_{2}, \dots,\mathbf{E}_{n}) \sim
      \mathcal{E}_{K \times nD}(\mathbf{0}, \mathbf{\Sigma}_{K} \otimes \mathbf{\Sigma}_{D} \otimes
      \mathbf{I}_{n}, h),
    $$
    where $\cov(\Vec \mathbb{E}^{T}) = \mathbf{\Sigma}_{K} \otimes \mathbf{\Sigma}_{D} \otimes
   \mathbf{I}_{n}$. Then proceeding in a similar way and generalising to no central case the results in
   \cite[Eq. 3.4.14, p. 109]{fz:90} and \cite[Theorem 5.1.6, p. 170]{gv:93}, we obtain that the joint
   distribution of $\mathbf{B}_{1}, \mathbf{B}_{2}, \dots,\mathbf{B}_{n}$ is
   $$
     \mathbb{B} = (\mathbf{B}_{1}, \mathbf{B}_{2}, \dots,\mathbf{B}_{n}) \sim
     \mathcal{GPW}_{K,n}\left(\mathbf{\Sigma}_{K}^{*},\frac{D}{2},\frac{D}{2},\dots,\frac{D}{2},
     \mathbf{\Omega}, h\right),
   $$
   where $\mathbf{\Omega} = \left(\mathbf{\Sigma}_{K}^{*}\right)^{-} \boldsymbol{\mu}^{*}
   \mathbf{\Sigma}_{D}^{-1}\boldsymbol{\mu}^{*T}$. Noting that in this case it is implicitly assuming that
   the sample $\mathbf{X}_{1}, \dots, \mathbf{X}_{n}$ is dependent.
  \item Alternatively, and recalling that the method-of-moments estimators are not uniquely defined (see Remark \ref{rem3})
  the method-of-moments estimator of $\mathbf{\Sigma}_{D}$ can be obtained from first two moments
  of $\mathbf{B}$ too.
\end{enumerate}

\section*{Acknowledgments}
This article was written under the existing research agreement between the first author and the
Universidad Aut\'onoma Agraria Antonio Narro, Saltillo, M\'exico. F. Caro was supported by a research
project of  University of Medell\'{i}n, Colombia.

\appendix

\section{Particular generalised Pseudo-Wishart singular distributions}

The following result is a particular case of the general result in \cite{dggf:05} or \cite{dggj:06}, when
$\mathbf{\Theta}$ is non singular and the notation of this paper is assumed.

\begin{theorem} [Generalised \ singular Pseudo-Wishart distributions]\label{teo31}
Assume that $\mathbf{Y} \sim {\mathcal{E}}_{K \times D}^{K-1,D}(\boldsymbol{\mu}, \mathbf{\Sigma} \otimes
\mathbf{\Theta},h)$, where $h$ admits a power series expansion
$$
   h(v + a) = \sum_{t = 0}^{\infty} \frac{h^{(t)}(a)v^{t}}{t!}.
$$
in $\Re$. Let, also, $q = \min(K-1,D)$; then the density of $\mathbf{B} = \mathbf{Y}
\mathbf{\Theta}^{-1}\mathbf{Y}^{T}$ is given by
\begin{equation}\label{wpwm}
   = \frac{\pi ^{qD/2} |\mathbf{L}|^{(D-K-1)/2}}{\Gamma_{q}(\Half D) \left (
  \displaystyle\prod_{i=1}^{(K-1)} \lambda_{i}^{D/2} \right )} \sum_{t = 0}^{\infty}\sum_{\kappa}^{}
  \frac{h^{(2t)}(\tr ( \mathbf{\Sigma}^{-}\mathbf{B} + \mathbf{\Omega}))}{t!} \frac{C_{\kappa}( \mathbf{\Omega}
  \mathbf{\Sigma}^{-} \mathbf{B})}{\left (\Half D \right)_{\kappa}}(d\mathbf{B})
\end{equation}
where $\mathbf{B} = \mathbf{W}_{1}\mathbf{L}\mathbf{W}^{T}_{1}$, is the nonsingular espectral
decomposition of $\mathbf{B}$ with $\mathbf{W}_{1}$ a semiorthogonal matrix, i.e.
$\mathbf{W}_{1}^{T}\mathbf{W}_{1} = \mathbf{I}_{q}$, and $\mathbf{L} = \diag(l_{1}, \dots,l_{q})$;
$\mathbf{\Omega} = \mathbf{\Sigma}^{-}\boldsymbol{\mu} \mathbf{\Theta}^{-1} \boldsymbol{\mu}^{T}$,
$(d\mathbf{B})$ is Hausdorff measure is defined in \cite[Section 5]{dggf:05}; $\lambda_{i}$, $i = 1,
\dots,(K-1)$, are nonull eigenvalues of $\mathbf{\Sigma}$, and where $C_{\kappa}(\mathbf{A})$ are the
zonal polynomials of $\mathbf{A}$ corresponding to the partition $\kappa = (t_{1}, \dots, t_{\alpha})$ of
$t$, with $\sum_{1}^{\alpha} t_{i} = t$; $(a)_{\kappa} = \prod_{j = 1}^{\alpha} (a-(j-1)/2)_{t_{j}}$,
$(a)_{t} = a(a+1) \cdots (a+t-1)$, being the generalized hypergeometric coefficients and $\Gamma_{s}(a) =
\pi^{s(s-1)/4} \prod_{j =1}^{s} \Gamma(a-(j-1)/2)$ is the multivariate gamma function, see \cite{Mh:82};
\end{theorem}

\begin{corollary}[Singular Pseudo-Wishart Gaussian  distribution]
Let us suppose that $\mathbf{Y} \sim {\mathcal{N}}_{K \times D}^{K-1,D}(\boldsymbol{\mu}, \mathbf{\Sigma}
\otimes \mathbf{\Theta})$, and let $q = \min(K-1,D)$; then the density of $\mathbf{B} = \mathbf{Y}
\mathbf{\Theta}^{-1}\mathbf{Y}^{T}$ is given by
\begin{equation}\label{ec:6}
   = C \etr\left(-\Half (\mathbf{\Sigma}^{-} \mathbf{B} -
   \mathbf{\Omega})\right) {}_{0}F_{1}\left(\Half D; \frac{1}{4} \mathbf{\Omega \Sigma}^{-} \mathbf{B}\right)
   (d\mathbf{B}),
\end{equation}
with
$$
  C = \frac{\pi^{D(q-(K-1))/2} |\mathbf{L}|^{(D-K-1)/2}}{2^{D(K-1)/2} \Gamma_{q}[\Half D]
   \left (\displaystyle \prod_{i=1}^{K-1} \lambda_{i}^{D/2} \right)},
$$
where ${}_{0}F_{1}(\cdot)$ is a hypergeometric function with a matrix argument, see \cite[p. 258]{Mh:82}.
\end{corollary}

\section{Singular Pseudo-Wishart Kotz distribution.}

Firs recall that the $K\times D$ random matrix $\mathbf{X}$ is said to have a \textit{singular matrix
multivariate symmetric Kotz type distribution} with parameters $N,r,s\in\Re$, $\boldsymbol{\mu}: K\times
D$, $\boldsymbol{\Sigma}:K\times K$, of rank $K-1$, $\boldsymbol{\Theta}:D\times D$ with $r>0$, $s>0$,
$2N+(K-1)D>2$, $\boldsymbol{\Sigma}>\mathbf{0}$, and $\boldsymbol{\Theta}>\mathbf{0}$ if its density is
$$
  \frac{sr^{(2N+(K-1)D-2)/2s}\Gamma\left((K-1)D/2\right)}{\pi^{(K-1)D/2}\Gamma\left[(2N+(K-1)D-2)/2s\right]
  \left (\displaystyle \prod_{i=1}^{K-1} \lambda_{i}^{D/2} \right)|\boldsymbol{\Theta}|^{(K-1)/2}}
$$
$$
 \times \  \left[\tr  \mathbf{\Theta}^{-1} (\mathbf{Y}-\boldsymbol{\mu})^{T}\mathbf{\Sigma}^{-}
    (\mathbf{Y}-\boldsymbol{\mu})\right]^{N-1}
  \exp\left\{-r\tr^{s} \mathbf{\Theta}^{-1} (\mathbf{Y}-\boldsymbol{\mu})^{T}\mathbf{\Sigma}^{-}
    (\mathbf{Y}-\boldsymbol{\mu})\right\}.
$$
When $T=s=1$, and $R=1/2$ we get the singular matrix variate gaussian distribution.

Note  that particular singular Pseudo-Wishart distributions just depend on the general derivative
$h^{(2t)}(\cdot)$ of the elliptical generator function; it seems a trivial fact, but the general formulae
involves cumbersome expressions indexed by partitions, see \cite{cdg:09}.  In the case of Kotz type
distribution they derived the following expressions.

When $s=1$,  the Kotz type models and their general derivative simplify substantially. Thus, the
following expressions applies for Gaussian, and the so called Kotz 1, Kotz 2, with parameters $N=1, s=1,
r=1/2$; $N=2, s=1, r=1/2$;  $N=3, s=1, r=1/2$; respectively.  The generator model is given by
$$
    h(y) = \frac{r^{N-1+(K-1)D/2}\Gamma\left[(K-1)D/2\right]}{\pi^{(K-1)D/2}
  \Gamma\left[N-1+(K-1)D/2\right]}y^{N-1}\exp\{-ry\},
$$

And, the corresponding $k$-th derivative of $h$, follows from
$$
  \displaystyle\frac{d^{k}}{dy^{k}}y^{N-1}\exp[-ry],
$$
which is given by
$$
  (-r)^{k}y^{N-1}\exp[-ry]\left\{1+\sum_{v=1}^{k}\binom{k}{v}
  \left[\prod_{i=0}^{v-1}(N-1-i)\right](-ry)^{-v}\right\},\nonumber\\
$$ where $k=2t$.

For the remaining models of the example, the so termed Kotz 3, Kotz 4 and Kotz 5, with parameters  $N=2,
s=2, r=1/2$; $N=2, s=3, r=1/2$ and $N=20, s=20, r=1/2$, respectively, the generator function is given by:
$$
  h(y)=\displaystyle\frac{s r^{(2N+(K-1)D-2)/2s} \ \Gamma[(K-1)D/2]}{\pi^{(K-1)D/2} \Gamma\left[(2N+(K-1)D-2)/2s \right] }
  \ y^{N-1}\exp\left(-r y^{s}\right)
$$
meanwhile the required $k$-th derivative of $h$, follows from $\frac{d^{k}}{dy^{k}}\exp\left(-r
y^{s}\right)$, which is given by
\begin{small}
$$
  y^{T-1}e^{-Ry^{s}}\left\{\sum_{\kappa\in P_{k}}\frac{k!(-R)^{\sum_{i=1}^{k}v_{i}}
  \prod_{j=0}^{k-1}(s-j)^{\sum_{i=j+1}^{k}v_{i}}}{\prod_{i=1}^{k}v_{i}!(i!)^{v_{i}}}y^{\sum_{i=1}^{k}(s-i)v_{i}}\right.
$$
$$
  + \ \sum_{m=1}^{k}\binom{k}{m} \left[\prod_{i=0}^{m-1}(T-1-i)\right]
$$
$$
  \left.\times\sum_{\kappa\in P_{k-m}}\frac{(k-m)!(-R)^{\sum_{i=1}^{k-m}v_{i}} \prod_{j=0}^{k-m-1}
  (s-j)^{\sum_{i=j+1}^{k-m}v_{i}}}{\prod_{i=1}^{k-m}v_{i}!(i!)^{v_{i}}}y^{\sum_{i=1}^{k-m}(s-i)v_{i}-m} \right\},
$$
\end{small}
where $\sum_{\kappa\in P_{k}}$ denotes the summation over all the partitions
$$
  \kappa=\left(k^{v_{k}},(k-1)^{v_{k}-1},\ldots,3^{v_{3},2^{v_{2}},1^{v_{1}}}\right)
$$
of $k$, with $\sum_{i=1}^{k}iv_{i}=k$, i.e. $\kappa$ is a partition of $k$ consisting of $v_{1}$ ones,
$v_{2}$ twos, $v_{3}$ threes, etc. It is important to quote that all the singular Pseudo-Wishart
distributions associated with the above Kotz type kernels can be computed by some modifications of the
algorithms provided by \cite{KE06} for the Gaussian case, see for example \cite{dc:13} and similar works
of the authors on shape theory.

\begin{thebibliography}{}

\bibitem[Arnold(1981)]{a:81}
    \textsc{Arnold, S. F.} (1981).
    Theory of linear models and multivariate analysis,
    John Wiley \& Sons, Inc.,
    New York.

\bibitem[Bookstein (1986)]{b:86}
    \textsc{Bookstein, F. L.} (1986).
    Size and shape spaces for landmark data in two dimensions (with discussion),
    \textit{Stat. Sci.} \textbf{1} 181--242.

\bibitem[Caro-Lopera and D\'{\i}az-Garc\'{\i}a(2012)]{CD12}
    \textsc{Caro-Lopera, F. C. and D\'{\i}az-Garc\'{\i}a, J. A.} (2012).
    Matrix Kummer-Pearson VII relation and polynomial pearson VII configuration density,
    \textit{J. Iranian Stat. Soc.} \textbf{11}(2) 217--230.

\bibitem[Caro-Lopera et al.(2009)]{cdg:09}
    \textsc{Caro-Lopera, F. J., D\'{\i}az-Garc\'{\i}a, J. A., and Gonz\'{a}lez-Far\'{\i}as, G.} (2009).
    Noncentral elliptical configuration density,
    \textit{J. Multivariate Anal.} \textbf{101}(1) 32--43.

\bibitem[Caro-Lopera et al (2014)]{CDG14}
    \textsc{Caro-Lopera, F. J., D\'{\i}az-Garc\'{\i}a, J. A., and Gonz\'{a}lez-Far\'{\i}as, G.} (2014).
    Inference in affine shape theory under elliptical models,
    \textit{J. Korean Stat. Soc.} \textbf{43}(1) 67--77.

\bibitem[Caro-Lopera et al.(2012)]{clb:12}
    \textsc{Caro-Lopera, F. J., Leiva, V. and Balakrishnan. N.} (2012).
    Connection between the Hadamard and matrix products with
    an application to matrix-variate Birnbaum-Saunders distributions,
    \textit{J. Multivariate Anal.} \textbf{104} 126--139.

\bibitem[D\'{\i}az-Garc\'{\i}a and Caro-Lopera(2012a)]{DC12a}
    \textsc{D\'{\i}az-Garc\'{\i}a, J. A., and Caro-Lopera, F. J.} (2012a).
    Generalised shape theory via SV decomposition I,
    \textit{Metrika} \textbf{75}(4) 541--565.

\bibitem[D\'{\i}az-Garc\'{\i}a and Caro-Lopera (2012b)]{dc:12}
    \textsc{D\'{\i}az-Garc\'{\i}a, J. A., and Caro-Lopera, F. J.} (2012b).
    Statistical theory of shape under elliptical models and singular value decompositions,
    \textit{J. Multivariate Anal.} \textbf{103}(1) 77–-92.

\bibitem[D\'{\i}az-Garc\'{\i}a and Caro-Lopera (2013)]{dc:13}
    \textsc{D\'{\i}az-Garc\'{\i}a, J. A., and Caro-Lopera, F. J.} (2013).
    Generalised shape theory via pseudo Wishart distribution,
    \textit{Sankhy$\bar{a}$ A} \textbf{75}(2) 253–-276.

\bibitem[D\'{\i}az-Garc\'{\i}a and Caro-Lopera(2014)]{DC14}
    \textsc{D\'{\i}az-Garc\'{\i}a, J. A., and Caro-Lopera, F. J.} (2014).
    Statistical theory of shape under elliptical models via QR decomposition,
    \textit{Statistics} \textbf{48}(2) 456--472.

\bibitem[D\'{\i}az-Garc\'{\i}a(1994)]{dg:94}
    \textsc{D\'{\i}az-Garc\'{\i}a, J. A.} (1994).
    \textit{Contributions to the theory of Wishart and multivariate elliptical distributions},
    Ph.D. disertation, Universidad de Granada, Spain, (in Spanish).

\bibitem[D\'{\i}az-Garc\'{\i}a and Gutiérrez Jáimez(1996)]{dg:96}
    \textsc{D\'{\i}az-Garc\'{\i}a, J. A., and R. Gutiérrez Jáimez, R.} (1996).
    \textit{Matrix differential calculus and moments of a random matrix elliptical},
    Serie Colección "Estadística Multivariable y Procesos Estocásticos".
    Universidad de Granada, España, (in Spanish).

\bibitem[D\'{\i}az-Garc\'{\i}a and Gonz\'alez-Far\'{\i}as (2005)]{dggf:05}
    \textsc{D\'{\i}az-Garc\'{\i}a, J. A., and Gonz\'alez-Far\'{\i}as, G.} (2005).
    Singular Random Matrix decompositions: Distributions,
    \textit{J. Multivariate Anal.} \textbf{94}(1) 109--122.

\bibitem[D\'{\i}az-Garc\'{\i}a \textit{et al}.(2003)]{dgr:03}
    \textsc{D\'{\i}az-Garc\'{\i}a, J. A., Guti\'errez, R. J. and Ramos-Quiroga, R.} (2003).
    Size-and-shape cone, shape disk and configuration densities for
    the elliptical models,
    \textit{Brazilian J. Prob. Stati.} \textbf{17} 135--146.

\bibitem[D\'{\i}az-Garc\'{\i}a and Guti\'errez-J\'aimez(2006)]{dggj:06}
    \textsc{D\'{\i}az-Garc\'{\i}a, J. A., and Guti\'errez-J\'aimez, R.} (2006).
    Wishart and Pseudo-Wishart distributions under elliptical laws and
    related distributions in the shape theory context,
    \textit{J. Statist. Plann. Inference} \textbf{136}(12) 4176--4193.

\bibitem[Dryden and Mardia(1998)]{DM98}
    \textsc{Dryden, I. L. and and Mardia, K. V.} (1998).
    \textit{Statistical shape analysis},
    John Wiley and Sons, Chichester.

\bibitem[Dryden et al.(2009)]{DKZ09}
    \textsc{Dryden, I. L., Koloydenko, A., and Zhou D.} (2009).
    Non-Euclidean statistics for covariance matrices, with applications to diffusion tensor imaging,
    \textit{An. App. Statist.} \textbf{3} 1102--1123.

\bibitem[Dutilleul(1999)]{d:99}
    \textsc{Dutilleul, P.} (1999).
    The mle algorithm for the matrix normal distribution,
    \textit{J. Statist. Comput. Simul.} \textbf{64} 105--123.

\bibitem[Fang and Zhang(1990)]{fz:90}
    \textsc{Fang, K. T. and Zhang, Y. T.} (1990).
    \textit{Generalized multivariate analysis},
    Science Press, Beijing, Springer-Verlang.

\bibitem[Fang, Kotz and Ng(1990)]{fkn:90}
    \textsc{Fang, K. T., Kotz, S., and Ng, K. W.} (1990).
    \textit{Symmetric multivariate and related distributions},
    Chapman Hall, London.

\bibitem[Goodall (1991)]{g:91}
    \textsc{Goodall, C. R.} (1991).
    Procustes methods in the statistical analysis of shape (with discussion),
    \textit{J. Royal Stat. Soc. B} \textbf{53} 285--339.

\bibitem[Goodall and Mardia (1993)]{gm:93}
    \textsc{Goodall, C. R.,  and Mardia,  K. V.} (1993).
    Multivariate aspects of shape theory,
    \textit{Ann.  Statist.} \textbf{21} 848--866.

\bibitem[Gupta and Varga(1993)]{gv:93}
    \textsc{Gupta, A. K.,  and  Varga, T.} (1993).
    \textit{Elliptically contoured models in statistics},
    Kluwer Academic Publishers, Dordrecht


\bibitem[Kendall(1984)]{ken:84}
    \textsc{Kent D. G.} (1984).
    Shape manifolds, Procrustean metrics and complex projective spaces,
    \textit{Bulletin of the London Mathematical Society}, \textbf{16}  81--121.

\bibitem[Kent(1992)]{k:92}
    \textsc{Kent J. T.} (1992).
    \textit{New directions in shape analysis},
    In Mardia, K. V., editor, The Art of Statistical Science,
    115-127. Wiley, Chichester.

\bibitem[Khatri(1968)]{k:68}
    Khatri, C. G. (1968).
    Some results for the singular normal multivariate regression nodels,
    \textit{Sankhy$\bar{a}$ A} \textbf{30}  267--280.

\bibitem[Koev and Edelman (2006)]{KE06}
    \textsc{Koev, P., and Edelman, A.} (2006).
    The efficient evaluation of the hypergeometric function of a matrix argument,
    \textit{Math. Comp.} \textbf{75} 833--846.

\bibitem[Le and Kendall(1993)]{LK93}
    \textsc{Le, H. L., and Kendall, D. G.} (1993).
    The Riemannian structure of Euclidean spaces: a novel environment for
    statistics,
    \textit{Ann. Statist.} \textbf{21} 1225--1271.

\bibitem[Lele(1991)]{l:91}
    \textsc{Lele, S.} (1991).
    Some Comments on Coordinate Free and Scale Invariant Method in Morphometrics,
    \textit{Am. J. Phys. Anthropol.} \textbf{85} 407--418.


\bibitem[Lele(1993)]{l:93}
    \textsc{Lele, S.} (1993).
    Euclidean distance matrix analysis (EDMA): Estimation of mean form and mean form difference,
    \textit{Math.  Geol.} \textbf{25}(5) 573--602.

\bibitem[Lele and Richtsmeier(1991)]{lr:91}
    \textsc{Lele, S., and Richtsmeier, J.} (1991).
    Euclidean distance matrix analysis: A coordinate free approach for comparing biological shapes using landmark data,
    \textit{Amer. J. Phys. Anthropol.} \textbf{86} 415--427.

\bibitem[Lele and Richtsmeier(1990)]{lr:90}
    \textsc{Lele, S., and Richtsmeier, J.} (1990).
    Statistical models in morphometrics: Are they realistic?
    \textit{Syst. Zool.} \textbf{39}(l) 60--69,.

\bibitem[Magnus and Neudecker(1979)]{mn:79}
    \textsc{Magnus, J. R., and Neudecker, H.} (1979).
    The commutation matrix: Some properties and applications,
    \textit{Ann. Statist.} \textbf{7}(2) 381--394.

\bibitem[Mardia and Dryden(1989)]{md:89}
    \textsc{Mardia, K. V., and Dryden, I. L.} (1989)
    The Statistical Analysis of Shape Data,
    \textit{Biometrika} \textbf{76} 271--281.


\bibitem[Mood \textit{et al}(1974)]{mgb:74}
    \textsc{Mood, A. M., Graybill, F. A., and Boes, D. C.} (1974).
    \textit{Introduction to the theory statistics},
    Third edition, McGraw-Hill Series in Probability and Statistics
    New York.

\bibitem[Muirhead (1982)]{Mh:82}
    \textsc{Muirhead, R. J.} (1982).
    \textit{Aspects of multivariate statistical theory},
    Wiley Series in Probability and Mathematical Statistics. John Wiley \& Sons, Inc.,
    New York.

\bibitem[Nadarajah(2003)]{k:03}
    \textsc{Nadarajah, S.} (2003).
    The Kotz-type distribution with applications,
    \textit{Statistics} \textbf{37}(4) 341--358.

\bibitem[Rao(1973)]{r:73}
    \textsc{Rao, C. R.} (1973).
    \textit{Linear Statistical Inference and Its Applications},
    John Wiley \& Sons, New York, 1973.

\bibitem[Richtsmeier \textit{et al.}(2002)]{rdl:02}
    \textsc{Richtsmeier, J. T., Deleon, V. B.,  and  Lele, S. R.} (2002).
    The Promise of Geometric Morphometrics,
    \textit{Yearbook of Phis. Anthropol.} \textbf{45} (2002) 63--91.

\bibitem[Walker(2001)]{w:01}
    \textsc{Walker, J.} (2001).
    Ability of geometric morphometric methods to estimate a known covariance matrix,
    \textit{Syst. Biol.} \textbf{49} 686-–696.

\end{thebibliography}
\end{document}